\documentclass[a4paper,english,oneside,11pt]{amsart}
\usepackage[BCOR0mm,DIV9]{typearea}
\usepackage[latin1]{inputenc}
\usepackage[T1]{fontenc}
\usepackage{lmodern}
\usepackage{dsfont}
\usepackage[all]{xy}
\usepackage{amsmath,amscd}
\usepackage{amsthm}
\usepackage{mathtools}
\usepackage{amssymb}
\usepackage{amsfonts}
\usepackage{latexsym}
\usepackage{pdfsync}
\usepackage{hyperref}
\xyoption{graph}

\title{Partial actions and KMS states on relative graph $\cs$-algebras}
\author[T. M. Carlsen]{Toke M. Carlsen}
\address{Department of Mathematical Sciences, Norwegian University of Science and Technology, N-7491 Trondheim, Norway}
\email{Toke.Meier.Carlsen@math.ntnu.no}
\author[N. S. Larsen]{Nadia S. Larsen}
\address{Department of Mathematics, University of Oslo, PO Box 1053, N-0316, Oslo, Norway}
\email{nadiasl@math.uio.no}
\date{\today}
\newtheorem{theorem}{Theorem}[section]
\newtheorem{lemma}[theorem]{Lemma}
\newtheorem{proposition}[theorem]{Proposition}
\newtheorem{corollary}[theorem]{Corollary}

\theoremstyle{definition}
\newtheorem{definition}[theorem]{Definition}
\newtheorem{example}[theorem]{Example}

\newtheorem{remark}[theorem]{Remark}

\newcommand{\cs}{C^*}
\newcommand{\cha}[1]{1_{#1}}

\newcommand{\inv}{^{-1}}

\newcommand\3[1]{{\mathds #1}}

\newcommand{\C}{\mathbb{C}}

\newcommand{\R}{\mathbb{R}}
\newcommand{\Z}{\mathbb{Z}}
\newcommand{\N}{\mathbb{N}}
\newcommand{\No}{{\N_{0}}}
\newcommand{\T}{\mathbb{T}}

\newcommand{\cyl}[1]{Z(#1)}
\newcommand{\cylf}[1]{Z_F(#1)}

\newcommand{\free}{\mathbb{F}}
\newcommand{\domain}[1]{U(#1\inv)}
\newcommand{\rang}[1]{U(#1)}
\newcommand{\pa}{\phi}
\newcommand{\tdomain}[1]{U(#1\inv)}
\newcommand{\trang}[1]{U(#1)}
\newcommand{\tpa}{\phi}
\newcommand{\cros}{C_0(\partial_R E)\rtimes_\Phi\free}
\newcommand{\spac}{\overline{\spa}}

\newcommand{\abs}[1]{\lvert #1\rvert}
\newcommand{\norm}[1]{\lVert #1\rVert}
\newcommand{\reg}{{\operatorname{reg}}}
\newcommand{\Cf}{D^{\beta}_{\operatorname{fin}}}
\newcommand{\Cinf}{D^{\beta}_{\operatorname{inf}}}
\newcommand{\Cinfa}{D^{\beta}_{\operatorname{con}}}
\newcommand{\Cinfb}{D^{\beta}_{\operatorname{dis}}}
\newcommand{\Eva}{E_a^*v}

\newcommand{\Evav}{vE_s^*v}
\newcommand{\Zva}{Z_{v}^a}
\newcommand{\Zvav}{Z_{v}^s}
\newcommand{\Erec}{E^\infty_{\operatorname{rec}}}
\newcommand{\Ereg}{E^0_{\beta\text{-reg}}}
\newcommand{\Ecrit}{E^0_{\beta\text{-crit}}}

\newcommand{\Eequ}{E^0_{\beta\text{-equ}}}
\newcommand{\Ewan}{E^\infty_{\operatorname{wan}}}
\newcommand{\Es}[2]{{#1}E_a^*{#2}}
\newcommand{\Winf}{W_{\operatorname{inf}}}
\newcommand{\conv}{\operatorname{Conv}}

\DeclareMathOperator{\aut}{Aut}
\DeclareMathOperator{\id}{Id}
\DeclareMathOperator{\spa}{span}

\numberwithin{equation}{section}

\begin{document}

\thanks{This research was partially supported by the Research Council of Norway through the project "Operator Algebras", and by the NordForsk research network "Operator Algebra and Dynamics" (grant \#11580).}

\begin{abstract}
	The relative graph $C^*$-algebras introduced by Muhly and Tomforde are generalizations of both graph algebras and their Toeplitz extensions. For an arbitrary graph $E$ and a subset  $R$ of the set of regular vertices of $E$ we show that the relative graph $C^*$-algebra $C^*(E, R)$ is isomorphic to a partial crossed product for an action of the free group generated by the edge set on the relative boundary path space. Given a time evolution on $C^*(E, R)$ induced by a function on the edge set, we characterize the KMS$_\beta$ states and ground states using an abstract result of Exel and Laca. Guided by their work on KMS states for Toeplitz-Cuntz-Krieger type algebras associated to infinite matrices, we obtain complete descriptions of the convex sets of  KMS states of finite type and of KMS states of infinite type whose associated measures are supported on recurrent infinite paths. This allows us to give a complete concrete description of the convex set of all KMS states for a big class of graphs which includes all graphs with finitely many vertices.
\end{abstract}
\maketitle

\section{Introduction}
\label{sec:introduction}

Characterizations of KMS$_\beta$ states and ground states on $C^*$-algebras of Toeplitz and Cuntz-Krieger type associated to a directed graph $E$ have been obtained in different contexts by many authors, see for example
 \cite{MR759450}, \cite{MR602475}, \cite{MR2057042},
\cite{MR1953065}, \cite{MR2056837}, \cite{MR500150} and \cite{MR1785460}. A classical reference for the definition of
KMS$_\beta$ states and  ground states as well as  background is \cite[Section 5.3]{MR1441540}. Recently there has been renewed interest in constructions of KMS states for graph algebras, see \cite{KW}, \cite{aHLRS} and \cite{Cas-Mor}. KMS weights on $C^*$-algebras associated to graphs were studied in \cite{MR2907004} and \cite{Tho}.

There are several techniques used in characterizing KMS states and constructing them in specific cases. These can employ the definition of the graph $C^*$-algebra $C^*(E)$ and its Toeplitz extension $\mathcal{T}C^*(E)$ as universal $C^*$-algebras with generators and relations, see e.g. \cite{aHLRS} and \cite{Cas-Mor}, or the realization of $C^*(E)$ and $\mathcal{T}C^*(E)$ as $C^*$-algebras of Pimsner type associated to a Hilbert bimodule, see \cite{MR2056837} and \cite{KW}, or the description of $C^*(E)$ as a groupoid $C^*$-algebra, see \cite{aHLRS}, which appeals to the general result from \cite{Nesh}.

A different type of general method that provides characterizations of KMS states was developed by Exel and Laca  in their study of Toeplitz-Cuntz-Krieger type algebras for infinite matrices, and uses crossed products by partial group actions, cf. \cite{MR1703078} and \cite{MR1953065}.
Our contribution here is an analysis of KMS and ground states on $C^*$-algebras $C^*(E, R)$ associated to an arbitrary directed graph $E$ and subsets $R$ of the regular vertices (i.e. those vertices that are not sinks or infinite emitters) by means of
realizing any such algebra as a partial crossed product for an action of the free group generated by the edge set. The space acted upon is a certain collection of boundary paths, and the resulting setup in the spirit of \cite{MR1953065} seems very well-suited for the  analysis of KMS states of the various algebras. By emphasizing the common picture of $C^*(E)$ and $\mathcal{T}C^*(E)$ as relative graph algebras in the sense of \cite{MT}, we obtain a unified  description of KMS states, see Theorem~\ref{thm:kms} for the precise statement. Not surprisingly, KMS states correspond to probability measures on the boundary path space that satisfy a certain  scaling property.

With motivation coming from the distinction between measures of finite and infinite type that is a crucial ingredient in \cite{MR1953065}, we distinguish between three classes of measures. First we have the finite type measures as in \cite{MR1953065}. Next, we identify as a new ingredient two kinds of infinite type measures: the measures that are supported on \emph{recurrent} paths, and the infinite type measures that are supported on  \emph{wandering} paths. Following  \cite{Tho2}, we call these \emph{conservative} and \emph{dissipative} measures, respectively. For the finite type and conservative measures we give a complete parametrization of the extreme points of the convex set of KMS$_\beta$ states for $\beta\in [0,\infty)$ in terms of what we call regular and critical vertices. A similar parametrization of the dissipative measures seems more difficult to obtain. We also provide a complete concrete description of the extreme points of the convex set of ground states, and furthermore we identify all those ground states that are KMS$_\infty$ states as introduced in \cite{Con-Mar}. We  show moreover that the
measures of infinite type correspond bijectively to normalized
eigenvectors of a (possibly infinite and even uncountable) matrix
associated to the directed graph under consideration.

 We illustrate in examples that several new phenomena occur in the configuration of KMS states for $\mathcal{T}C^*(E)$ and $C^*(E)$ endowed with the gauge action in case that $E$ is an infinite graph. In Example~\ref{ex:a} we present an infinite strongly connected graph  $E$  of finite degree (or valence) for which $C^*(E)$ has no KMS states, thus showing that the analogue of \cite[Theorem 4.3]{aHLRS} does not hold for infinite graphs. In Example~\ref{ex:motivating} we show that on $\mathcal{T}C^*(E)$, all three types of states can occur;  moreover, the finite type and infinite type states co-exist on a  critical interval, and there is a "phase-transition" between the two infinite types at a critical temperature.

The proof of the main general result relies on the
characterization of KMS states on the crossed product of a
semi-saturated orthogonal partial action of a free group on a
$\cs$-algebra obtained by Exel and Laca
\cite[Theorem 4.3]{MR1953065}. Towards using the Exel-Laca result, for a given directed graph $E$ and a subset $R$ of the regular vertices, we start by constructing a locally compact space $\partial_R E$ and a semi-saturated orthogonal partial action $\Phi$ of the free group $\free$ generated by $E^1$ on the
$\cs$-algebra $C_0(\partial_R E)$, see Section~\ref{sec:part-action-on-graph}. We point out that the space
$\partial E$ corresponding to the choice of $R$ as the entire subset  of regular vertices of $E$ is the boundary path space of the graph, which recently played a role in \cite{W}. In Section \ref{sec:cs-algebra-graph}  we prove that the crossed product $C_0(\partial_R E)\rtimes_\Phi \free$ is isomorphic to $C^*(E, R)$, see Theorem~\ref{theorem:partial}. A similar result is obtained with different methods in \cite{MR1452280} in the case of the usual Cuntz algebra $\mathcal{O}_n$ and Cuntz-Krieger algebra $\mathcal{O}_A$ and their Toeplitz extensions.

Our main general characterization of KMS and ground states on $C^*(E, R)$ is contained in Theorem~\ref{thm:kms}. This result has close connections to the existing  literature, and we elaborate on this point in several remarks.  In section~\ref{section:extreme} we introduce the sets
$\Ereg$ and $\Ecrit$ of regular and critical vertices for a given $\beta\in[0,\infty)$, associate measures to them, and develop a machinery that enables us to give the characterizations of the convex sets of measures of finite type and of conservative measures, see Theorems~\ref{thm:finitetype-reg} and \ref{thm:infinitetype-crit}. Section~\ref{section:ground} deals with ground states and KMS$_\infty$ states. In the final section we present several examples, all of which illustrate in different ways that a much richer structure of KMS states can be expected as one passes from finite to infinite graphs.

We thank M. Laca for suggesting, at a very early stage of this project, to look more carefully at the Toeplitz extension of the graph $C^*$-algebra, and K. Thomsen for valuable comments to an earlier draft of this paper.

\section{A partial action on the relative path space of a graph}
\label{sec:part-action-on-graph}

Let $E=(E^0,E^1,s,r)$ be a directed graph: by this we mean that $E^0$ and $E^1$ are
arbitrary (not necessarily countable) sets and $s$ and $r$ are maps from $E^1$ to $E^0$. Elements of $E^0$ are called
\emph{vertices} of $E$ and elements of $E^1$ are \emph{edges} of $E$. If $e\in E^1$, then $s(e)$ denotes the
\emph{source} of $e$ and $r(e)$ the \emph{range} of $e$.
A \emph{path of length $n$} in $E$ is a sequence $e_1e_2\dots e_n$ of
edges in $E$ such that $r(e_i)=s(e_{i+1})$ for
$i\in\{1,2,\dotsc,n-1\}$ (the reader should be warned that in some papers and books the roles of $r$ and
$s$ are interchanged, so a path would be a sequence $e_1e_2\dots e_n$ of
edges in $E$ such that $s(e_i)=r(e_{i+1})$ for
$i\in\{1,2,\dotsc,n-1\}$). We regard vertices as paths of length 0 and
edges as paths of length 1. We denote by $E^n$ the set of paths of length
$n$ in $E$ and write $E^*$ for the set $\bigcup_{n\in\No}E^n$. We
write $|u|$ for the length of a path $u\in E^*$, and we let $E^{\le n}$ be the collection of paths $u$
with $|u|\le n$. We extend the range
and source maps to $E^*$ by setting $s(u)=s(e_1)$ and $r(u)=r(e_n)$
for $u=e_1e_2\dots e_n\in E^n$ with $n\ge 1$, and $s(v)=r(v)=v$ for
$v\in E^0$. If $v\in E^0$, then we let $vE^n=\{u\in E^n\mid s(u)=v\}$ and $E^nv=\{u\in E^n\mid r(u)=v\}$, $vE^*=\{u\in E^*\mid s(u)=v\}$ and $E^*v=\{u\in E^*\mid r(u)=v\}$. We let $E^0_\reg=\{v\in E^0\mid vE^1\text{ is finite and non-empty}\}$.
If $u=e_1e_2\dots e_n$ and $u'=e'_1e'_2\dots e'_m$ are
paths with $r(u)=s(u')$, then we write $uu'$ for the path $e_1e_2\dots
e_ne'_1e'_2\dots e'_m$ obtained from concatenation of the two paths.

We recall from, for example \cite{MR1988256}, \cite{MR2135030} and
\cite{MR1914564} that the $\cs$-algebra $\cs(E)$ of the graph $E$ is defined as the
universal $\cs$-algebra generated by a \emph{Cuntz-Krieger $E$-family}  $(s_e,p_v)_{e\in
  E^1,v\in E^0}$ consisting of partial isometries $(s_e)_{e\in E^1}$
with mutually orthogonal range projections
and mutually orthogonal projections $(p_v)_{v\in E^0}$ satisfying
\begin{enumerate}\renewcommand{\theenumi}{CK\arabic{enumi}}
\item\label{it:CK1} $s_e^*s_e=p_{r(e)}$ for all $e\in E^1$,
\item\label{it:CK2} $s_e s_e^*\le p_{s(e)}$ for all $e\in E^1$,
\item\label{it:CK3} $p_v=\sum_{e\in vE^1}s_e s_e^*$ for $v\in E^0_\reg$.
\end{enumerate}

In \cite{MT}, \emph{relative graph $C^*$-algebras} were introduced. To define a relative graph $C^*$-algebra we must in addition to a directed graph $E$ specify a subset $R$ of $E^0_\reg$.
The \emph{relative graph $\cs$-algebra} $\cs(E,R)$ of the graph $E$ relative to $R$ is then defined as the universal $\cs$-algebra generated by a \emph{Cuntz-Krieger $(E,R)$-family} $(s_e,p_v)_{e\in
  E^1,v\in E^0}$ consisting of partial isometries $(s_e)_{e\in E^1}$
with mutually orthogonal range projections
and mutually orthogonal projections $(p_v)_{v\in E^0}$ satisfying \eqref{it:CK1} and \eqref{it:CK2} above plus the relation
\begin{enumerate}\renewcommand{\theenumi}{RCK3}
\item\label{it:RCK3} $p_v=\sum_{e\in vE^1}s_e s_e^*$ for $v\in R$.
\end{enumerate}

\begin{remark} \textnormal{(a)} If $R=E^0_\reg$, then a Cuntz-Krieger $(E,R)$-family is the same as a Cuntz-Krieger $E$-family and, consequently,  $C^*(E,R)=C^*(E)$.

\textnormal{(b)} If $R=\emptyset$, we claim that $C^*(E,R)$ is the Toeplitz algebra $\mathcal{T}C^*(E)$ introduced in \cite{FR}. Indeed, in this case the relation \eqref{it:CK3} is vacuous, and \eqref{it:CK2} in connection with the assumption that the ranges of the $s_e$'s  are mutually orthogonal imply that
\begin{equation}\label{eq:Toeplitz-relation}
p_v\geq \sum_{e\in F}s_e s_e^*
\end{equation}
whenever $F$ is a finite subset of $vE^1$. As remarked in \cite[Lemma 1.1]{aHLRS}, the converse holds, thus \eqref{eq:Toeplitz-relation} alone implies that the ranges of the $s_e$'s are mutually orthogonal.
\end{remark}

It follows from the universal properties of $C^*(E,R)$ that there is a strongly continuous action (the \emph{gauge action}) $\gamma:\T\to\aut(C^*(E,R))$ satisfying for all $z\in\T$ that $\gamma_z(p_v)=p_v$ for $v\in E^0$, and $\gamma_z(s_e)=zs_e$ for $e\in E^1$.

In order to obtain a different picture of $C^*(E,R)$ we now turn to the path space and the boundary path space. An \emph{infinite path} in $E$ is an infinite sequence $e_1e_2\dots$ of
edges in $E$ such that $r(e_i)=s(e_{i+1})$ for
$i\in\N$. We write $E^{\infty}$ for the set of infinite paths in $E$.

The \emph{path space} of the graph is
$E^{\le\infty}:=E^*\cup E^\infty$.
When $R$ is a subset of $E^0_\reg$, the \emph{relative boundary path space} $\partial_R E$ is defined by
\begin{equation*}
  \partial_R E=E^\infty\cup\{u\in E^*\mid r(u)\notin R\}.
\end{equation*}
Equivalently, $\partial_R E=E^{\le\infty}\setminus \{u\in E^*\mid r(u)\in R\}$.
If $R=E^0_\reg$, then the relative boundary path space $\partial_R E$ is also called the \emph{boundary path space} and is denoted by $\partial E$, see e.g. \cite{W}.

Given $u=e_1e_2\dots e_n\in E^n$ and $0\leq m\leq n$, we let $u(0,m)$ denote the path $e_1e_2\dots e_m $ if
$1\le m\le n$, and $s(u)$ if $m=0$. Likewise, for an infinite path $x=e_1e_2\dots\in
E^\infty$ and $m\ge 1$, we denote $x(0,m)$ the path $e_1e_2\dots e_m$, and if $m=0$, then we write $x(0,m)$
or $s(x)$ for the vertex $s(e_1)$. Extending our convention for concatenation of finite paths, if
$u=e_1e_2\dots e_n\in E^n$
and $x=e'_1e'_2\dots\in E^\infty$ with $r(u)=s(x)$, then we write $ux$
for the resulting infinite path.

In order to compare finite paths with arbitrary paths, we introduce the following notation: for $u\in E^*$ and $x\in E^{\le\infty}$
we write $u\le x$ if
$|u|\le |x|$ and $x(0,|u|)=u$. Further, we write $u<x$ if $u\le x$ and $u\ne
x$. For $u\in E^*$ the \emph{cylinder set of
  $u$}  $\cyl{u}$ is defined by
\begin{equation}\label{def:cylinder-set}
  \cyl{u}=\{x\in E^{\le\infty}\mid u\le x\}.
\end{equation}
We denote by $\mathcal{F}(E^*)$ the collection of finite subsets of the space $E^*$. Given $u\in E^*$ and
$F\in \mathcal{F}(E^*)$, we let
\begin{equation}\label{def:ZFu}
\cylf{u}=\cyl{u}\setminus\Bigl(\bigcup_{\underset{u\leq u'}{u'\in F}}\cyl{u'}\Bigr);
\end{equation}
note that this is the empty set precisely when $u\in F$. In particular, if $Z_F(u)$ is non-empty, then it contains $u$.

The statements \eqref{item:16}-\eqref{item:19} in the next result can be found in \cite{W}.  We include them in order to have
the necessary terminology at hand for proving \eqref{item:20}-\eqref{item:22}.

\begin{proposition} \label{prop:et}
  Let $E$ be a directed graph and endow the path space $E^{\le\infty}$
  with the topology generated by $\{\cyl{u}\mid u\in
    E^*\}\cup\{E^{\le\infty}\setminus\cyl{u}\mid u\in
    E^*\}$. We then have:
  \begin{enumerate}\renewcommand{\theenumi}{\roman{enumi}}
  \item\label{item:16} $E^{\le\infty}$ is a totally disconnected
    locally compact Hausdorff topological space.
  \item\label{item:17} $E^{\le\infty}$ is compact if and only if $E^0$ is finite.
  \item\label{item:18} The system $
      \bigl\{\cylf{u}\mid
      u\in E^*,\ F \in \mathcal{F}(E^*)\bigr\}$
    is a basis of open and compact subsets for the topology of
    $E^{\le\infty}$.
  \item\label{item:19} If $u\in E^*$, then the system $
    \bigl\{\cylf{u}\mid
    F \in  \mathcal{F}(E^*)
    \text{ such that }\cylf{u}\ne\emptyset\bigr\}$
    is a neighbourhood basis for $u$.
  \item\label{item:20} Every $u\in E^*$ for which $r(u)E^1$ is a finite set, is isolated.
  \item\label{item:21} $E^*$ is dense in $E^{\le\infty}$.
  \item\label{item:22} The closure of $E^\infty\cup\{u\in E^*\mid
    r(u)E^1=\emptyset\}$ in $E^{\le\infty}$ is $\partial E$.
  \end{enumerate}
\end{proposition}

\begin{proof}
 \eqref{item:20}: If $u\in E^*$, then $\{u\}=\cyl{u}\setminus\Bigl(\bigcup_{e\in r(u)E^1}\cyl{ue}\Bigr)$.
  Thus if $r(u)E^1$ is a finite set $F$, then $\{u\}=\cylf{u}$, so $u$ is isolated.

  \eqref{item:21}: Given $x\in E^\infty$, suppose $x\in Z_F(u)$ for some $u\in E^*$ and $F\in \mathcal{F}(E^*)$ with
  $u\le u'$ for every $u'\in F$. Then $x(0,\vert u\vert)\in Z_F(u)$, hence $u\in Z_F(u)\cap E^*$.
  It follows that
  $E^*$ is dense in $E^{\le\infty}$.

  \eqref{item:22}: If $u\in E^*$, $r(u)E^1$ is infinite, and $F$ is a finite subset of
  $E^*$ such that $u\le u'$ for every $u'\in F$, then there exists
  at least one $e\in r(u)E^1$ such that $ue\in
  \cylf{u}$, and thus at least one
  element in $E^\infty\cup\{u'\in E^*\mid
  s\inv(r(u'))=\emptyset\}$ which belongs to
  $\cylf{u}$. Hence $\partial E$ is
  contained in the closure of $E^\infty\cup\{u\in
  E^*\mid r(u)E^1=\emptyset\}$. That $\partial E$ is closed follows from
  the fact that every $u\in E^{\le\infty}\setminus\partial E$ is such that $r(u)E^1$ is finite (and non-empty),
  and therefore $u$ is isolated by \eqref{item:20}.  This proves \eqref{item:22}.
\end{proof}

\begin{corollary} \label{cor:rps}
	Let $E$ be a directed graph and let $R$ be a subset of $E^0_\reg$. Equip the relative boundary path space $\partial_R E$ with the topology it inherits from $E^{\le\infty}$ when the latter is given the topology described in Proposition \ref{prop:et}. Then $\partial_R E$ is a totally disconnected locally compact Hausdorff topological space.
\end{corollary}

\begin{proof}
According to Proposition \ref{prop:et}\eqref{item:16}, $E^{\le\infty}$ is a totally disconnected locally compact Hausdorff topological space. If $u\in E^{\le\infty}\setminus \partial_RE$, then $u\in E^*$ and $r(u)\in R\subseteq E^0_\reg$, so $u$ is isolated according to \ref{prop:et}\eqref{item:20}. It follows that $\partial_R E$ is a closed subset of $E^{\le\infty}$ and therefore a totally disconnected locally compact Hausdorff topological space.
\end{proof}

Let $E$ be a directed graph and $R$ a subset of $E^0_\reg$. Let $\free$ denote the free group generated by $E^1$. An edge $e\in E^1$ will still be denoted $e$
as an element of $\free$, and $e^{-1}$ will denote the inverse of $e$ in $\free$.
We shall view $E^*\setminus E^0$ as a subset of $\free$
upon identifying a path $u=e_1e_2\dots e_n$ with the element in $\free$ obtained by multiplication of $e_1,\dots, e_n$. The identity
 element of $\free$ will be denoted $\31$.  An element $g\in\free$ is in \emph{reduced form} if $g=a_na_{n-1}\dotsm a_1$ for $a_1,a_2,\dotsc,a_n\in E^1\cup\{e\inv\mid e\in
E^1\}$ such that $a_i\ne a_{i+1}\inv$ whenever $i\in\{1,2,\dotsc,n-1\}$. We denote
$|g|$ the number of generators in the reduced form of $g$. Notice that this use of $|u|$ agrees with the previously defined
use of $|u|$ as the length of an element $u\in E^*\setminus E^0$.

Now we construct a semi-saturated and orthogonal partial action of $\free$ on $\partial_RE$. We will do this by defining open and compact subsets $\tdomain{g}$ and $\trang{g}$ of $\partial_RE$ together with a homeomorphism $\tpa_g$ taking $\tdomain{g}$ onto $\trang{g}$ and satisfying the axioms of a partial action as $g\in \free$.

First, let $\tpa_{\31}$ denote the identity map on $\partial_RE$ and let
$\tdomain{\31}=\trang{\31}=\partial_RE$. For $e\in E^1$, let
$\trang{e}=\cyl{e}\cap \partial_RE$ and $\tdomain{e}=\cyl{r(e)}\cap \partial_RE$, and define maps
\begin{align}
\tpa_e:\tdomain{e}\to \partial_RE,\ &\tpa_e:x\mapsto ex,\label{eq:phi-e}\\
\tpa_{e\inv}:\trang{e}\to \partial_RE,\  &\tpa_{e\inv}:ex\mapsto x.\label{eq:phi-einv}
\end{align}
Let $g=a_na_{n-1}\dotsm a_1\in\free$ be in reduced form. We will define $\tdomain{(a_ia_{i-1}\dotsm a_1)}$ and $\tpa_{a_ia_{i-1}\dotsm a_1}:\tdomain{(a_ia_{i-1}\dotsm a_1)} \to \partial_RE$
for all $i\in\{1,2,\dots,n\}$ recursively. For $i=1$, $\tdomain{a_1}$, $\trang{a_1}$ and $\tpa_{a_1}:\tdomain{a_1}\to \partial_RE$ have already been defined. For
$i>1$, we let
\begin{gather}
\tdomain{(a_ia_{i-1}\dotsm a_1)} = \tpa_{a_{i-1}\dotsm a_1}\inv(\tdomain{a_i}) \text{ and}\label{eq:def-Ug}\\
\tpa_{a_ia_{i-1}\dotsm a_1}(x) = \tpa_{a_i}(\tpa_{a_{i-1}\dotsm a_1}(x))\text{ for }x\in \tdomain{(a_ia_{i-1}\dotsm a_1)}.\label{eq:phi-ai}
\end{gather}

For later use, we record some properties of the sets $U(g)$ and the maps $\phi_g$.

\begin{lemma}\label{lem:alternative-U-phi} The sets $U(g)$ and the maps $\phi_g$ for $g\in \free$ defined by \eqref{eq:def-Ug}-\eqref{eq:phi-ai} satisfy the following:
\begin{enumerate}\renewcommand{\theenumi}{\roman{enumi}}
  \item \label{item:g1} $\domain{\31}=\rang{\31}=\partial_R E$ and $\pa_{\31}(x)=x$ for $x\in \domain{\31}$.
		\item \label{item:g2} If $u\in E^*\setminus E^0$, then $\domain{u}=\cyl{r(u)}\cap\partial_R E$, $\rang{u}=\cyl{u}\cap\partial_R E$ and $\pa_u(x)=ux$ for $x\in \domain{u}$.
		\item \label{item:g3} If $u=e_1\dots e_m,u'=e'_1\dots e'_n\in E^*\setminus E^0$, $e_m\ne e'_n$ and $r(u)=r(u')$, then $\rang{u'u\inv}=\cyl{u'}\cap\partial_R E$ and $\pa_{u(u')\inv}(u'x)=ux$ for $x\in\rang{u'u\inv}$.
		\item \label{item:g4} If $g$ does not belong to $\{\31\}\cup\{u\mid u\in E^*\setminus E^0\}\cup\{u\inv\mid u\in E^*\setminus E^0\}\cup\{u(u')\inv\mid u,u'\in E^*\setminus E^0,\ r(u)=r(u')\}$, then $\domain{g}=\rang{g}=\emptyset$.
  \end{enumerate}
\end{lemma}

\begin{proof}
	Claim \eqref{item:g1} follows directly from the definition of $\tdomain{\31}$, $\trang{\31}$ and $\tpa_{\31}$, and \eqref{item:g2} and \eqref{item:g3} follow easily from \eqref{eq:phi-e}--\eqref{eq:phi-ai}.
	
	To prove \eqref{item:g4}, notice that if $g$ does not belong to $\{\31\}\cup\{u\mid u\in E^*\setminus E^0\}\cup\{u\inv\mid u\in E^*\setminus E^0\}\cup\{u(u')\inv\mid u,u'\in E^*\setminus E^0,\ r(u)=r(u')\}$, then the reduced form of $g$ either contains a factor of the form $e\inv e'$ with $e,e'\in E^1$ and $e\ne e'$, a factor of the form $ee'$ with $e,e'\in E^1$ and $r(e)\ne s(e')$, or a factor of the form $(e')\inv e\inv $ with $e,e'\in E^1$ and $s(e)\ne r(e')$. It follows from \eqref{eq:phi-e}--\eqref{eq:phi-ai} that $\tdomain{g}=\trang{g}=\emptyset$ in each of these cases.
\end{proof}

\begin{proposition} \label{prop:action-on-path-space}
  Let $E$ be a directed graph and let $R$ be a subset of $E^0_\reg$. Equip the relative boundary path space $\partial_R E$ with the topology described in Corollary \ref{cor:rps}, let $\free$ be the free group generated by $E^1$, and let for each $g\in\free$ the set $\tdomain{g}$ and the map $\tpa_g:\tdomain{g}\to \partial_R E$ be as defined above.  We then have:
	\begin{enumerate}\renewcommand{\theenumi}{\roman{enumi}}
		\item \label{item:g5} For each $g\in\free$, the set $\rang{g}$ is an open and compact subset of $\partial_R E$ and $\pa_g$ is a homeomorphism from $\domain{g}$ onto $\rang{g}$ with inverse $\pa_{g\inv}$.
		\item \label{item:g6}  $\pa_{g_1}(\domain{g_1}\cap\rang{g_2})=\rang{g_1}\cap\rang{g_1g_2}$ for $g_1,g_2\in\free$.
		\item \label{item:g7} $\pa_{g_1}\circ\pa_{g_2}=\pa_{g_1g_2} \text{ on } \domain{g_2}\cap\domain{(g_1g_2)}$ for $g_1,g_2\in\free$.
		\item \label{item:g8} $\Phi=(\pa_g)_{g\in\free}$ is a semi-saturated and orthogonal partial action of $\free$ on $\partial_R E$ as in \cite[Section 2]{MR1703078}.
	\end{enumerate}
\end{proposition}

\begin{proof} Item \eqref{item:g5} follows directly from \eqref{item:g1}--\eqref{item:g4} in Lemma~\ref{lem:alternative-U-phi}
and Proposition \ref{prop:et}.
	
	For \eqref{item:g6}, notice first that \eqref{eq:def-Ug}, \eqref{eq:phi-ai} and \eqref{item:g5} imply that if $a_n\dotsm a_1$ is an element of $\free$ in reduced form, then
\begin{equation}\label{eq:alternative-phi}
\trang{a_n\dotsm a_1}=\tpa_{a_n\dotsm a_i}(\tdomain{(a_n\dotsm a_i)}\cap\trang{a_{i-1}\dotsm a_1})
 \end{equation}
 for any $i\in\{2,\dots,n\}$. If $g_1=a_m\dotsm a_1$ and $g_2=a_n'\dotsm a_1'$ are in reduced forms and we let $i$ be the largest nonnegative integer such that $a_1\inv\dotsm a_i\inv=a_n'\dotsm a_{n+1-i}'$ (with $i=0$ if $a_1\inv\ne a_n'$), then $a_m\dotsm a_{i+1}a_{n-i}'\dotsm a_1'$ is the reduced form of $g_1g_2$ (where $a_m\dotsm a_{i+1}=\31$ if $i=m$, and $a_{n-i}'\dotsm a_1'=\31$ if $i=n$). It follows from \eqref{eq:alternative-phi} that
	\begin{align*}
		\tdomain{g_1}\cap\trang{g_2}&=\trang{a_1\inv\dotsm a_m\inv}\cap\trang{a_n'\dotsm a_1'}\\
		&=\tpa_{a_1\inv\dotsm a_i\inv}(\trang{(a_i\dotsm a_1}\cap\trang{a_{i+1}\inv\dotsm a_m\inv}\cap\trang{a_{n-i}'\dotsm a_1'}).
	\end{align*}
	Once again using \eqref{eq:alternative-phi}, we have
	\begin{align*}
		\tpa_{g_1}(\tdomain{g_1}\cap\trang{g_2})&=\tpa_{a_m\dotsm a_{i+1}}(\trang{a_i\dotsm a_{1}}\cap\trang{a_{i+1}\inv\dotsm a_m\inv}\cap\trang{a_{n-i}'\dotsm a_1'})\\
		&=\trang{a_m\dotsm a_1}\cap\trang{a_m\dotsm a_{i+1}a_{n-i}'\dotsm a_1'} =\trang{g_1}\cap\trang{g_1g_2}.
	\end{align*}
	
	For \eqref{item:g7}, notice that it follows from \eqref{item:g5} and \eqref{item:g6} that $\tdomain{g_2}\cap\tdomain{(g_1g_2)}=\tpa_{g_2\inv}(\trang{g_2}\cap\trang{g_1})=\tpa_{g_2}\inv(\trang{g_1})$, so $\tpa_{g_1g_2}$ and $\tpa_{g_1}\circ\tpa_{g_2}$ are defined on the same subsets of $E^{\le\infty}$. If $g_1=a_m\dotsm a_1$, $g_2=a_n'\dotsm a_1'$ and $i$ are as above, then \eqref{eq:phi-ai} and \eqref{item:g5} imply that
	\begin{align*}
		\tpa_{g_1}\circ\tpa_{g_2}(x)&=\tpa_{a_m}\circ\dots\circ\tpa_{a_1}\circ\tpa_{a_n'}\circ\dots\circ\tpa_{a_1'}(x)\\
		&=\tpa_{a_m}\circ\dots\circ\tpa_{a_{i+1}}\circ\tpa_{a_{n-i}'}\circ\dots\circ\tpa_{a_1'}(x)=\tpa_{g_1g_2}(x)
	\end{align*}
	for $x\in\tdomain{g_2}\cap\tdomain{(g_1g_2)}$ proving \eqref{item:g7}.
	
	It follows from \eqref{item:g5}--\eqref{item:g7} that   $\Phi=(\tpa_g)_{g\in\free}$ is a partial action of $\free$ on $\partial_RE$. If $e,f\in E^1$ and $e\ne f$, then $\trang{e}\cap\trang{f}=\cyl{e}\cap\cyl{f}\cap\partial_RE=\emptyset$. Furthermore, if $g_1=a_m\dotsm a_1$ and $g_2=a_n'\dotsm a_1'$ are in reduced forms and $\abs{g_1g_2}=\abs{g_1}+\abs{g_2}$, then $a_m\dotsm a_1a_n'\dotsm a_1'$ is the reduced form of $g_1g_2$, so $\trang{g_1g_2}=\tpa_{g_1}(\tdomain{g_1}\cap\trang{g_2})\subset\trang{g_1}$. It therefore follows from \cite[Proposition 4.1]{MR1953065} that $\Phi$ is semi-saturated and orthogonal, showing the claim in  \eqref{item:g8}.
\end{proof}

\section{The graph algebras $C^*(E)$  and $\mathcal{T}C^*(E)$ as partial crossed products}\label{sec:cs-algebra-graph}

In this section we show that $C^*(E,R)$ can be realised as the full crossed product of the partial action $(C_0(\partial_R E), \free, \Phi)$ introduced in the previous section.

We start by introducing the terminology we need. There are different definitions of both a partial
representation of a discrete group on a Hilbert space (in e.g. \cite{MR1953065} and
\cite{MR1452280}) and of a covariant representation of a partial dynamical system $(A, G, \alpha)$ (in  \cite{MR1331978} and \cite{MR1452280}). However, reassuring equivalences of these definitions were shown in \cite{MR1452280}, and we refer to \cite[\S 1]{ELQ} for a brief but illuminating overview of the main concepts and constructions.

Suppose that $(A, G, \alpha)$ is a
partial dynamical system: thus for each $g\in G$  the maps $\alpha_g$ are $*$-isomorphisms between closed, two-sided ideals
$D_{g^{-1}}$ and $D_g$ of a $C^*$-algebra $A$ such that $\alpha_e=\id_A$ and $\alpha_{gh}$ extends $\alpha_g\alpha_h$ for all $g,h\in G$. The \emph{full crossed product} $A\rtimes_\alpha G$ is the enveloping $C^*$-algebra of the convolution $*$-algebra $\{f\in l^1(G, A)\mid f(g)\in D_g\}$ endowed with a suitable norm.

A partial representation of a group $G$ on a Hilbert space $H$ is a map $u:G\to B(H)$ such that  $u(e)=1$ (where here, $e$ denotes the neutral element of the group $G$),
$u_{g^{-1}}=u_g^*$ and $u_gu_hu_{h^{-1}}=u_{gh}u_{h^{-1}}$ for all $g,h\in G$, see \cite{MR1953065}. A \emph{covariant representation} of  $(A, G, \alpha)$ is a pair $(\pi,u)$ that consists of a nondegenerate representation of $A$ on a Hilbert space $H$ and a partial representation $u$ of
$G$ on $H$ such that for $g\in G$ we have
\begin{align}
	&u_gu_g^* \text{ is the projection onto the subspace } \spac\pi(D_g)H \text{ and}\label{p1}\\
	&u_g\pi(f)u_{g\inv}=\pi(f\circ\alpha_{g\inv}) \text{ for } f\in D_{g^{-1}}.\label{p2}
\end{align}

As pointed out by Quigg and Raeburn in \cite{MR1452280}, there is a universal covariant representation $(\iota, \delta)$ in the double dual $(A\rtimes_\alpha G)^{**}$ such that $A\rtimes_\alpha G$ is the closed linear span of finite sums of the form $\sum\iota(a)\delta_g$. Since $\delta_e=1$, it follows that $\iota(A)\subseteq  A\rtimes_\alpha G$, but we do not have in general that $\delta_g\in A\rtimes_\alpha G$. Since any partial system $(A,G, \alpha )$ admits covariant representations $(\pi,u)$ with $\pi$ faithful (e.g. the reduced covariant representation of \cite[Section 3]{MR1331978}), the representation $\iota$ is faithful.

Now we turn our attention back to the partial action of $\free$ on $\partial_R E$. In the sequel we regard $C_0(\trang{g})$ for $g\in\free$  as
an ideal of $C_0(\partial_RE)$ by letting $f(x)=0$ for $f\in
C_0(\trang{g})$ and $x\in \partial_RE\setminus \trang{g}$. If $f\in
C_0(\tdomain{g})$, then $f\circ\tpa_{g\inv}$ will denote the element of
$C_0(\trang{g})$ defined by
\begin{equation*}
  f\circ\tpa_{g\inv}(x)=
  \begin{cases}
    f(\tpa_{g\inv}(x))&\text{if }x\in\trang{g},\\
    0&\text{if }x\in \partial_RE\setminus\trang{g}.
  \end{cases}
\end{equation*}
Then the map $f\mapsto f\circ\tpa_{g\inv}$  is  a $*$-isomorphism
from $C_0(\tdomain{g})$ to $C_0(\trang{g})$. By a slight abuse of notation (which should not lead to any confusion) we
let $\tpa_g$ be the map $f\mapsto f\circ\tpa_{g\inv}$, and let $D_{g\inv}=C_0(\tdomain{g})$ be its domain while $D_g=C_0(\trang{g})$
is its range. Thus $\Phi=(\tpa_g)_{g\in\free}$ is a partial action of $\free$ on the $C^*$-algebra $C_0(\partial_RE)$.
The action is still semi-saturated and orthogonal.

We let $(\iota, \delta)$ denote the universal covariant representation of $(C_0(\partial_R E),\free,\Phi)$.
We recall from \cite[Theorem 4.3]{MR1953065} that given any function $N:E^1\to (1,\infty)$ there exists a unique strongly continuous one-parameter group $\sigma$ of automorphisms of $C_0(\partial_RE)\rtimes_{\Phi}\free$ such that
\begin{equation}\label{eq:sigma-from-N}
  \sigma_t(b)=\bigl(N(e)\bigr)^{it}b\,\text{ and }\,\sigma_t(c)=c
\end{equation}
for all $e\in E^1$, all $b\in\iota\bigl(C_0(\rang{e})\bigr)\delta_e$, and all $c\in\iota\bigl(C_0(\partial_RE)\bigr)$.

If  $N(e)=\exp(1)$ for every $e\in E^1$, then\footnote{Since $\exp(1)$ only makes a couple of appearances here, while elements in $E^1$ are used all throughout, we decided to give preference to the notation $e\in E^1$.} $\sigma_t$ is
$2\pi$-periodic, and so induces a strongly continuous action $\beta:\T\to\aut(C_0(\partial_RE)\rtimes_{\Phi}\free)$
such that $\beta(\delta_e)=z\delta_e\text{ and }\beta(f)=f$
for all $z\in\T$, $e\in E^1$, and $f\in C_0(\partial_RE)$.

When $U$ is a closed and open subset of $\partial_RE$, then $\cha{U}$ will denote the characteristic function of $U$.

\begin{theorem} \label{theorem:partial}
  Let $E$ be a directed graph and let $R$ be a subset of $E^0_\reg$. Let $\free$ be the free group generated by $E^1$, and let $\Phi$ be the partial action of $\free$ on $C_0(\partial_R E)$ described above. We then have:

\textnormal{(a)}\label{item:t1} There is a unique $*$-isomorphism $\rho:C^*(E,R)\to C_0(\partial_R E)\rtimes_{\Phi}\free$ which maps $p_v$ to $\iota(\cha{\cyl{v}\cap\partial_R E})$ for $v\in E^0$, and $s_e$ to $\delta_e$ for $e\in E^1$.
		
\textnormal{(b)}\label{item:t2} $\rho\circ\gamma_z=\beta_z\circ\rho$ for all $z\in\T$.
\end{theorem}

To prove Theorem \ref{theorem:partial}, we need the following lemma.

\begin{lemma}\label{lem:span-partialcp}
Each $\delta_e$, $e\in E^1$, belongs to $C_0(\partial_R E)\rtimes_{\Phi}\free$, and
$C_0(\partial_R E)\rtimes_{\Phi}\free$ is generated by the union $\{\iota(\cha{\cyl{v}\cap\partial_RE})\mid v\in E^0\}\cup\{\delta_e\mid e\in E^1\}$.
\end{lemma}

\begin{proof}
	Notice first that \eqref{p1} gives that
	\begin{equation}\label{eq:rangeproj}
	  \delta_g\delta_g^*=\iota(\cha{\trang{g}})
	  \end{equation}
	  for $g\in\free\setminus\{\31\}$. Thus if $e\in E^1$, then $\delta_e=\delta_e\delta_e^*\delta_e=\iota(\cha{\trang{e}})\delta_e\in C_0(\partial_R E)\rtimes_{\Phi}\free$.
	
  Since $\free$ is
  generated by $E^1$ and $\Phi$ is multiplicative
  (cf. \cite[Section 5]{MR1452280}), it follows that $\{\delta_g\mid
  g\in\free\setminus\{\31\}\}$ is contained in the
  $\cs$-algebra generated by $\{\delta_e\mid e\in E^1\}$. By
  the Stone-Weierstrass Theorem and Proposition \ref{prop:et}, the $C^*$-algebra
   $C_0(\partial_RE)$ is generated
  by $\{\cha{\cyl{u}\cap\partial_RE}\mid u\in E^*\}$, and since
  $\iota(\cha{\cyl{u}\cap\partial_RE})=\delta_u\delta_u^*$ for $u\in
  E^*\setminus E^0$, we get that $\iota(C_0(\partial_RE))$ is contained
  in the $\cs$-algebra generated by
  $\{\iota(\cha{\cyl{v}\cap\partial_RE})\mid v\in E^0\}\cup\{\delta_e\mid e\in E^1\}$.
	It follows that $C_0(\partial_R E)\rtimes_{\Phi}\free$ is generated by $\{\iota(\cha{\cyl{v}\cap\partial_RE})\mid v\in E^0\}\cup\{\delta_e\mid e\in E^1\}$.
\end{proof}

\begin{proof}[Proof of Theorem \ref{theorem:partial}](a):  By \eqref{eq:rangeproj} we have  that $\delta_e\delta_e^*=\iota(\cha{\rang{e}})=\iota(\cha{\cyl{e}\cap\partial_R E})$ and $\delta_e^*\delta_e=\iota(\cha{\domain{e}})=\iota(\cha{\cyl{r(e)}\cap \partial_R E})$ for every $e\in E^1$. Note that
\begin{equation}\label{eq:Zv-inclusion}
\bigcup_{e\in vE^1}\cyl{e}\cap\partial_R E \subseteq \cyl{v}\cap\partial_R E
\end{equation}
for all $v\in E^0_\reg$, where equality holds only when $v\in R$ (if $v\notin R$, then $v$ belongs to the right but not to the left-hand side because our convention is that $vE^1\subseteq E^1$).
It follows that the union $\{\iota(\cha{\cyl{v}\cap\partial_R E})\}_{v\in E^0}\cup\{\delta_e\}_{e\in E^1}$ is a Cuntz-Krieger $(E,R)$-family. Thus there exists a $*$-homomorphism $\rho$ from $\cs(E,R)$ to $\cros$ which for every $e\in E^1$ maps $s_e$ to $\delta_e$ and for every $v\in E^0$ maps $p_v$ to $\iota(\cha{\cyl{v}\cap\partial_R E})$. That this $*$-homomorphism is unique, follows from the fact that $C^*(E,R)$ is generated by $\{p_v\mid v\in E^0\}\cup\{s_e\mid e\in E^1\}$.

	According to\footnote{Notice that there is an obvious misprint in 2. of \cite[Theorem 3.11]{MT}} Theorem 3.11 of \cite{MT}, $\rho$ is injective if $\rho(p_v)\ne 0$ for every $v\in E^0$, $\rho(p_v-\sum_{e\in vE^1}s_es_e^*)\ne 0$ for $v\in E^0_\reg\setminus R$, and there exists an action $\beta:\T\to\aut(\cros)$ such that $\rho\circ\gamma_z=\beta_z\circ\rho$ for all $z\in\T$. Since $\iota$ is injective, it follows that $\rho(p_v)=\iota(\cha{\cyl{v}\cap\partial_RE})\ne 0$ for every $v\in E^0$, and that
\begin{align*}
\rho(\sum_{e\in vE^1}s_es_e^*)&=\sum_{e\in vE^1}\delta_e\delta_e^*=\sum_{e\in vE^1}\iota(\cha{\cyl{e}\cap\partial_RE})\\
&<\iota(\cha{\cyl{v}\cap\partial_RE})=\rho(p_v)
\end{align*}
for every $v\in E^0_\reg\setminus R$, where the inequality sign is from \eqref{eq:Zv-inclusion}. Thus the injectivity of $\rho$ will follow once we have proved part (b). That $\rho$ is surjective follows directly from Lemma \ref{lem:span-partialcp}.

Towards proving (b), let $z\in\T$. Since $C^*(E,R)$ is generated by $\{p_v\mid v\in E^0\}\cup\{s_e\mid e\in E^1\}$, $\rho(\gamma_z(p_v))=\rho(p_v)=\beta_z(\rho(p_v))$ for every $v\in E^0$, and $\rho(\gamma_z(s_e))=z\rho(s_e)=\beta_z(\rho(s_e))$ for every $e\in E^1$, it follows that $\rho\circ\gamma_z=\beta_z\circ\rho$, as wanted.
\end{proof}

\begin{remark}
A different proof of the claim that $\rho$ in Theorem~\ref{theorem:partial} is an isomorphism can be provided by directly constructing the inverse: this will be given by a covariant pair for $(C_0(\partial_R E), \free, \Phi)$  obtained from a partial representation of $\free$ inside $C^*(E, R)$. The details are similar to
the proof of \cite[Proposition 4.1]{MR1703078}.
\end{remark}

In the following we will write $s_u$ for $s_{e_1}s_{e_2}\dotsm s_{e_n}$ when $u=e_1e_2\dotsm e_n\in E^*\setminus E^0$, and $s_v$ for $p_v$ when $v\in E^0$.

\begin{corollary} \label{coro:inclusion}
  Let $E$ be a directed graph and $R$ a subset of $E^0_\reg$. Then:
	\begin{enumerate}\renewcommand{\theenumi}{\roman{enumi}}
		\item \label{item:c1} There exists a unique $*$-isomorphism $\tilde{\iota}$ from $C_0(\partial_R E)$ onto the $C^*$-subalgebra of $C^*(E,R)$ generated by $\{s_us_u^*\mid u\in E^*\}$ mapping $\cha{\cyl{u}\cap\partial_R E}$ to $s_us_u^*$ for every $u\in E^*$.
	\item \label{item:c2} There exists a unique norm-decreasing linear map (conditional expectation) $F$ from $C^*(E,R)$ onto the $C^*$-subalgebra of $C^*(E,R)$ generated by $\{s_us_u^*\mid u\in E^*\}$ such that
	 \begin{equation*}
    F(s_us_{u'}^*)=
    \begin{cases}
      s_us_{u'}^*
      &\text{if }u=u',\\
      0&\text{if }u\ne u'
    \end{cases}
  \end{equation*}
for $u,u'\in E^*$.
	\end{enumerate}
\end{corollary}

\begin{proof}
	\eqref{item:c1}: The map $\rho\inv\circ\iota$ is a $*$-isomorphism from $C_0(\partial_RE)$ onto the $C^*$-subalgebra of $C^*(E,R)$ generated by $\{s_us_u^*\mid u\in E^*\}$ mapping $\cha{\cyl{u}\cap\partial_R E}$ to $s_us_u^*$ for every $u\in E^*$. That this is the only $*$-isomorphism with this property follows from the fact that $C_0(\partial_RE)$, according to the Stone-Weierstrass Theorem, is generated by $\{\cha{\cyl{u}\cap\partial_RE}\mid u\in E^*\}$.
	
	\eqref{item:c2}: The existence of $F$ follows from Theorem \ref{theorem:partial} and \cite[Proposition 2.3]{MR1953065}. Uniqueness follows from the fact that $C^*(E,R)=\spac\{s_us_{u'}^*\mid u,u'\in E^*\}$.
\end{proof}

\section{KMS states on $C^*(E,R)$}
\label{sec:kms-states-graph}

We will in this section describe the sets of KMS states of certain one-parameter groups of automorphisms of $C^*(E,R)$ in terms of states of $C_0(\partial_R E)$, in terms of regular Borel probability measures on $\partial_R E$, and in terms of functions from $E^0$ to $[0,1]$.

We start by recalling the notions of KMS$_\beta$ states and ground states.  For the first one, a standard definition
is found in \cite{MR1441540} and \cite{Ped}. However, an  equivalent formulation has in recent times prevailed: given a
$C^*$-algebra $A$ and a homomorphism (a dynamics) $\sigma:\R\to \aut(A)$, an element $a\in A$
is called \emph{analytic} provided that $t\mapsto \sigma_t(a)$ extends to an entire function on $\C$. The analytic elements form a
dense subset of $A$, see \cite[\S 8.12]{Ped}. For $\beta\in (0, \infty)$, a \emph{KMS$_\beta$-state} of $(A,\sigma)$ is a state $\psi$ of $A$ which satisfies the KMS$_\beta$ condition
\begin{equation}\label{def:KMS-beta}
\psi(ab)=\psi(b\sigma_{i\beta}(a))
\end{equation}
for all $a,b$ analytic in $A$. It is known that it suffices to have \eqref{def:KMS-beta} satisfied for a
subset of analytic elements of $A$ that spans a dense subalgebra of $A$, \cite[Proposition 8.12.3]{MR1441540}. A \emph{KMS$_0$-state} of $(A,\sigma)$ is a state $\psi$ of $A$ which is invariant with respect to $\sigma$ (i.e., $\psi(\sigma_t(a))=\psi(a)$ for $t\in\R$ and $a\in A$), and which satisfies the trace condition $\psi(ab)=\psi(ba)$ for all $a,b\in A$.
A state $\psi$ on $A$ is a \emph{ground state} of $(A, \sigma)$ if for every $a,b$ analytic in $A$, the entire function $z\mapsto \psi(a\sigma_z(b))$ is bounded on the upper-half plane. Again,  it is known that it suffices to have boundedness for a set of elements that spans a dense subalgebra of the analytic elements.
%More recently, the
%notion of KMS$_\infty$-state was introduced in \cite{Con-Mar} and refers to states which are, by definition, weak$^*$-limits of
%KMS$_\beta$-states as $\beta$ runs over a net $\beta_i\to \infty$.

We will now describe the set of KMS states for certain one-parameter groups of automorphisms of $C^*(E,R)$.

 Suppose $N$ is a function $N:E^1\to (1,\infty)$, and let $\sigma$ be the  unique strongly continuous one-parameter group of automorphisms of $\cros$ given by \eqref{eq:sigma-from-N}. Theorem \ref{theorem:partial} therefore implies that $N$ gives rise to a unique strongly continuous one-parameter group $\sigma$ of automorphisms of $C^*(E,R)$ such that
\begin{equation*}
  \sigma_t(s_e)=\bigl(N(e)\bigr)^{it}s_e \text{ and }\sigma_t(p_v)=p_v
\end{equation*}
for all $e\in E^1$ and $v\in E^0$.

Before we state the result, we introduce some notation. For $0\leq\beta<\infty$ we define the following sets:
\begin{itemize}
		\item[$A^\beta$:] the set of KMS$_\beta$ states for $(C^*(E,R),\sigma)$,
		\item[$B^\beta$:] the set of states $\omega$ of $C_0(\partial_RE)$ that satisfy the scaling condition $\omega(f\circ\phi_e\inv)=\bigl(N(e)\bigr)^{-\beta}\omega(f)$ for every $e\in E^1$ and every $f\in C_0(\tdomain{e})$,
	  \item[$C^\beta$:] the set of regular Borel probability measures $\mu$ on $\partial_R E$ that satisfy the scaling condition $\mu(\phi_e(A))=N(e)^{-\beta}\mu(A)$ for every $e\in E^1$ and every Borel measurable subset $A$ of $\tdomain{e}$, and
	  \item[$D^\beta$:] the set of functions $m:E^0\to [0,1]$ such that
	    \begin{enumerate}\renewcommand{\theenumi}{m\arabic{enumi}}
	    	\item\label{item:m1} $\sum_{v\in E^0}m(v)=1$;
		    \item\label{item:m2} $m(v)=\sum_{e\in vE^1}\bigl(N(e)\bigr)^{-\beta}m\bigl(r(e)\bigr)$ if $v\in R$;
		    \item\label{item:m3} $m(v)\ge\sum_{e\in vE^1}\bigl(N(e)\bigr)^{-\beta}m\bigl(r(e)\bigr)$ for $v\in E^0$.
	    \end{enumerate}
	\end{itemize}
Note that \eqref{item:m1} is equivalent to $\sup\{\sum_{v\in F}m(v)\mid F\text{ is a finite subset of }E^0\}=1$ and \eqref{item:m3} to the assertion that $m(v)\ge\sum_{e\in F}\bigl(N(e)\bigr)^{-\beta}m\bigl(r(e)\bigr)$ for every finite subset $F$ of $vE^1$. Notice also that if $R=E^0$, then $D^\beta$ is the set of positive normalized eigenvectors with eigenvalue 1 of the matrix $(\sum_{e\in v'E^1v}N(e))_{v',v\in E^0}$ (where $v'E^1v=\{e\in E^1\mid s(e)=v',\ r(e)=v\}$). In particular, if $R=E^0$ and $N(e)=\exp(1)$ for all $e\in E^1$, then $D^\beta$ is the set of positive normalized eigenvectors with eigenvalue $\exp(\beta)$ of the adjacency matrix of $E$.

Further, we let $A^{\operatorname{gr}}$ be the set of ground states for $(C^*(E,R),\sigma)$, $B^{\operatorname{gr}}$ the set of states $\omega$ of $C_0(\partial_R E)$ such that $\omega\bigl(\cha{\rang{e}}\bigr)=0$ for every $e\in E^1$, $C_3^{\operatorname{gr}}$ the set of regular Borel probability measures $\mu$ on $\partial_R E$ that satisfy that $\mu(A)=0$ for every $e\in E^1$ and every Borel measurable subset $A$ of $\rang{e}$, and finally $D^{\operatorname{gr}}$ the set of functions $m:E^0\to [0,1]$ that satisfy
    \begin{enumerate}
    \item $\sum_{v\in E^0}m(v)=1$,
    \item $m(v)=0$ for $v\in R$.
    \end{enumerate}

\begin{theorem} \label{thm:kms}
  Given a directed graph $E$, a subset $R$ of $E^0_\reg$ and a function $N:E^1\to (1,\infty)$, let $\sigma$ be the strongly continuous one-parameter group
   of automorphisms of $C^*(E,R)$ such that
  \begin{equation*}
	  \sigma_t(s_e)=\bigl(N(e)\bigr)^{it}s_e \text{ and }\sigma_t(p_v)=p_v
	\end{equation*}
  for all $e\in E^1$ and $v\in E^0$. Further, let $\phi_g:\domain{g}\to\rang{g},\ g\in\free$ be the partial action of the free group  $\free$ generated by $E^1$ from Proposition~\ref{prop:action-on-path-space}.

  Then for  $\beta\in [0,\infty)$, $A^\beta$, $B^\beta$, $C^\beta$ and $D^\beta$ are isomorphic as convex sets. Likewise,
  $A^{\operatorname{gr}}$, $B^{\operatorname{gr}}$, $C^{\operatorname{gr}}$ and $D^{\operatorname{gr}}$ are isomorphic as convex sets.
\end{theorem}

Theorem \ref{thm:kms} will follow from Propositions \ref{prop:state}, \ref{prop:int} and \ref{prop:function} below. We point out that these propositions give explicit isomorphisms.

\begin{remark}
If the graph $E$ is finite, then Propositions \ref{prop:state} and \ref{prop:function} recover \cite[Proposition 2.1(a),(b),(c)]{aHLRS} when $R=\emptyset$ and \cite[Proposition 2.1(d)]{aHLRS} when $R=E^0_\reg$.
\end{remark}

\begin{remark}
In \cite{Cas-Mor}, KMS states on graph $C^*$-algebras are studied. The one-parameter group of automorphisms considered in \cite{Cas-Mor} is of the same form as the one-parameter group of automorphisms considered here, but in \cite{Cas-Mor} it is not required that $N(e)>1$ for all $e\in E^1$, only that there exists a $c>0$ such that $N(e)>c$ for all $e\in E^1$, and that $N(e_1)\cdots N(e_n)\ne 1$ for all $e_1\cdots e_n\in E^n$, $n\ge 1$. For $R=E^0_\reg$ and $\beta>0$, \cite[Theorem 3.3]{Cas-Mor} generalizes the results about $A^\beta$ and $B^\beta$ given in Proposition \ref{prop:int}, and \cite[Theorem 3.10]{Cas-Mor} generalizes the results about $B^\beta$ and $D^\beta$ given in Proposition \ref{prop:function}. For ground states, Proposition \ref{prop:int} recovers \cite[Proposition 4.3]{Cas-Mor} and Proposition \ref{prop:function} recovers \cite[Theorem 4.4]{Cas-Mor} when $R=E^0_\reg$.

We believe that with some effort, the results about $A^\beta$, $B^\beta$, $C^\beta$, and $D^\beta$ given in Theorem \ref{thm:kms}, Proposition \ref{prop:state}, Proposition \ref{prop:int}, and Proposition \ref{prop:function} could be generalized to the case where the requirement $N(e)>1$ for all $e\in E^1$ is replaced with the assumption that $N(e_1)\dots N(e_n)\ne 1$ for all $e_1\dots e_n\in E^n$, $n\ge 1$ (cf. the remark after the proof of Lemma 3.2 in \cite{MR1953065} and Remark \ref{rmk:groupoid}).
\end{remark}

\begin{remark}\label{rmk:groupoid}
Our characterization of KMS$_\beta$ states in Theorem~\ref{thm:kms} can be seen in relation to the general result for groupoid algebras in \cite{Nesh} because the $C^*$-algebra $C^*(E, R)$ admits a realization as a groupoid $C^*$-algebra $C^*(\mathcal{G}_{(E,R)})$ where $\mathcal{G}:=\mathcal{G}_{(E,R)}=\{(ux,\vert u\vert-\vert u'\vert,u'x)\mid u,u'\in E^*,x\in \partial_R E, r(u)=s(x)=r(u')\}$, see for example \cite{pat} and \cite[\S 6.4]{aHLRS}.

Assume $E$ is countable and let $N:E^1\to (1,\infty)$ be a function such that $N(e_1)\cdots N(e_n)\ne 1$ for all $e_1\cdots e_n\in E^n$, $n\ge 1$. Extend the function $N$ to $E^*$ by letting $N(v)=1$ for $v\in E^0$ and by letting $N(u)=N(u_1)\dotsm N(u_n)$ for $u=u_1\dotsm u_n\in E^n$ and $n\geq 1$. Define $c:\mathcal{G}\to \R$ by $c((ux,\vert u\vert-\vert u'\vert,u'x))=\ln N(u)-\ln N(u')$. Then $c$ is a continuous one-cocycle. For $x\in\partial_RE$ let $\mathcal{G}_x^x$ be the stabilizer $\{(x,n,x)\in\mathcal{G}\mid n\in\Z\}$. Then $\mathcal{G}_x^x$ is a subgroup of $\Z$. Let $(u_g)_{g\in\mathcal{G}_x^x}$ be the generators of the $C^*$-algebra $C^*(\mathcal{G}_x^x)$ of $\mathcal{G}_x^x$. Define the dynamics $\sigma^c$ on $\mathcal{G}$ by $\sigma^c_t(f)(g)=e^{itc(g)}f(g)$ for  $f\in C_c(\mathcal{G})$ and $g\in \mathcal{G}$. Then \cite[Theorem 1.3]{Nesh} provides, for all $\beta\in \R$, a one-to-one correspondence between $\sigma^c$-KMS$_\beta$ states on $C^*(E, R)$ and pairs $(\mu, \{\varphi_x\}_{x})$ consisting of a probability measure $\mu$ on the unit space with Radon-Nikodym cocycle $e^{-\beta c}$ and a measurable field of states $\varphi_x$, each defined on $C^*(\mathcal{G}_x^x)$ and satisfying $\phi_x(u_g)=\phi_{r(h)}(u_{hgh^{-1}})$ and $\phi_x(u_{g'})=0$ for $\mu$-a.e. $x$, all $g\in\mathcal{G}_x^x$, all $h\in\mathcal{G}_x$ and all $g'\in\mathcal{G}_x^x\setminus c^{-1}(0)$. Notice that a probability measure on the unit space with Radon-Nikodym cocycle $e^{-\beta c}$ is the same as an element of our $C^\beta$. If $(x,n,x)\in\mathcal{G}_x^x$, then $c((x,n,x))\ne 0$ unless $n=0$ (because of our assumption that $N(u)\ne 1$ unless $u\in E^0$). It follows that  there is just one state on $C^*(\mathcal{G}_x^x)$ satisfying that $\phi_x(u_{g'})=0$ for $g'\in\mathcal{G}_x^x\setminus c^{-1}(0)$ (cf. \cite[\S 6.4]{aHLRS}). Thus \cite[Theorem 1.3]{Nesh} gives the equivalence between  $A^\beta$ and $C^\beta$ from Theorem~\ref{thm:kms}.
\end{remark}

\begin{remark}
	Our Theorem~\ref{thm:kms} at $\beta=0$ generalizes one result from \cite{MR1991743}: the functions $m$ in $D^0$ are the  graph-traces in \cite{MR1991743}, and the bijective correspondence between tracial states on $C^*(E)$ and graph traces on $E$ under the assumption that the graph $E$ satisfies condition (K) is contained in Theorem~\ref{thm:kms}.
\end{remark}

\begin{proposition} \label{prop:state}
	In the setting of Theorem \ref{thm:kms}, let $\tilde\iota$ and $F$ be as in Corollary \ref{coro:inclusion}. Then $\omega\mapsto\omega\circ {\tilde\iota}\inv\circ F$ defines a convex isomorphism from $B^\beta$ to $A^\beta$ for $\beta\in [0,\infty)$, and a convex isomorphism from $B^{\operatorname{gr}}$ to $A^{\operatorname{gr}}$.
\end{proposition}

\begin{proof}
	For ground states and for $\beta>0$, the result follows directly from Theorem 4.3 of
	\cite{MR1953065}, Theorem \ref{theorem:partial} and Corollary \ref{coro:inclusion}. It remains to prove the case $\beta=0$, which  comes down to characterizing $\sigma$-invariant traces on $C^*(E, R)$.
	
	If $\omega\in B^0$, then $\psi:=\omega\circ {\tilde\iota}\inv\circ F$ is a $\sigma$-invariant state. Since $C^*(E,R)=\spac\{s_{u_1}s_{u_2}^*\mid u_1,u_2\in E^*\}$, it suffices to show that $\psi(s_{u_1}s_{u_2}^*s_{u_3}s_{u_4}^*)=\psi(s_{u_3}s_{u_4}^*s_{u_1}s_{u_2}^*)$ for $u_1,u_2,u_3,u_4\in E^*$ in order to prove that $\psi\in A^0$. We extend the definition of $\domain{u}$ and $\rang{u}$ to all $u\in E^*$ by letting $\domain{v}=\rang{v}=\cyl{v}\cap\partial_RE$ for $v\in E^0$. We then have that $\psi(s_{u_1}s_{u_2}^*s_{u_3}s_{u_4}^*)=0$ unless either
	\begin{enumerate}
		\item[{}] $u_2=u_3u$ and $u_1=u_4u$ for some $u\in E^*$, in which case $\psi(s_{u_1}s_{u_2}^*s_{u_3}s_{u_4}^*)=\psi(s_{u_1}s_{u_1}^*)=\omega\bigl(\cha{U(u_1)}\bigr)=\omega\bigl(\cha{\domain{u_1}}\bigr)=\omega\bigl(\cha{\domain{u}}\bigr)$, or
		\item[{}] $u_3=u_2u$ and $u_4=u_1u$ for some $u\in E^*$, in which case $\psi(s_{u_1}s_{u_2}^*s_{u_3}s_{u_4}^*)=\psi(s_{u_4}s_{u_4}^*)=\omega\bigl(\cha{U(u_4)}\bigr)=\omega\bigl(\cha{\domain{u_4}}\bigr)=\omega\bigl(\cha{\domain{u}}\bigr)$.
	\end{enumerate}
	Similarly, $\psi(s_{u_3}s_{u_4}^*s_{u_1}s_{u_2}^*)=0$ unless either
	\begin{enumerate}
		\item[{}] $u_2=u_3u$ and $u_1=u_4u$ for some $u\in E^*$, in which case $\psi(s_{u_3}s_{u_4}^*s_{u_1}s_{u_2}^*)=\psi(s_{u_2}s_{u_2}^*)=\omega\bigl(\cha{U(u_2)}\bigr)=\omega\bigl(\cha{\domain{u_2}}\bigr)=\omega\bigl(\cha{\domain{u}}\bigr)$, or
		\item[{}] $u_3=u_2u$ and $u_4=u_1u$ for some $u\in E^*$, in which case $\psi(s_{u_1}s_{u_2}^*s_{u_3}s_{u_4}^*)=\psi(s_{u_4}s_{u_4}^*)=\omega\bigl(\cha{U(u_4)}\bigr)=\omega\bigl(\cha{\domain{u_4}}\bigr)=\omega\bigl(\cha{\domain{u}}\bigr)$.
	\end{enumerate}
	Thus, $\omega\mapsto\omega\circ {\tilde\iota}\inv\circ F$ is a map from $B^0$ to $A^0$. It is clear that it is a convex map and that it is injective.
	
	Let $\psi\in A^0$. It follows from the $\sigma$-invariance of $\psi$ that $\psi(s_{u}s_{u'}^*)=0$ unless $\abs{u}=\abs{u'}$; in case  $\abs{u}=\abs{u'}$, then it follows from the trace property of $\psi$ that $\psi(s_{u}s_{u'}^*)=\psi(s_{u'}^*s_{u})=0$ unless $u=u'$, because $s_u$ and $s_{u'}$ have orthogonal range projections. Thus $\omega:=\psi\circ{\tilde\iota}$ is a state of $C_0(\partial_RE)$ such that $\omega\circ {\tilde\iota}\inv\circ F=\psi$. Let $e\in E^1$. If $u\in r(e)E^*$, then
	\begin{multline} \label{eq:inv}
		\omega(\cha{\cyl{u}\cap\partial_RE}\circ\phi_e\inv)=
		\omega(\cha{\cyl{eu}\cap\partial_RE})=
		\psi(s_{eu}s_{eu}^*)\\=
		\psi(s_e^*s_es_us_u^*)=
		\psi(s_us_u^*)=
		\omega(\cha{\cyl{u}\cap\partial_RE}).
	\end{multline}
	Since $C_0(\domain{e})=\spac\{\cha{\cyl{u}\cap\partial_RE}\mid u\in r(e)E^*\}$, the calculations \eqref{eq:inv} show that $\omega\in B^0$. Thus, $\omega\mapsto\omega\circ {\tilde\iota}\inv\circ F$ is surjective and therefore a convex isomorphism from $B^0$ to $A^0$.
\end{proof}

\begin{lemma}\label{lem:riesz}
  Let $E$ be a directed graph, $R$ a subset of $E^0_\reg$, and let $M:E^1\to [0,\infty)$ be a function. Then
  the map
  \begin{equation}
    \label{eq:4}
    \mu\mapsto \left(f\mapsto \int f\ d\mu\right)
  \end{equation}
  is a bijective correspondence between the set of regular Borel probability
  measures $\mu$ on $\partial_RE$ satisfying that $\mu(\tpa_e(A))=M(e)\mu(A)$
  for all $e\in E^1$ and all Borel measurable subsets $A$ of
  $\tdomain{e}$, and the set of states $\eta$ of
  $C_0(\partial_RE)$ satisfying that
  $\eta(f\circ\tpa_e\inv)=M(e)\eta(f)$ for all $e\in E^1$ and all $f\in C_0\bigl(\tdomain{e}\bigr)$.
\end{lemma}

\begin{proof}
  It follows from Riesz' Representation Theorem (see for example
  \cite[6.16]{MR918770}) that \eqref{eq:4} is a bijective
  correspondence between the set of regular Borel probability measures
  on $\partial_RE$ and the set of states $\eta$ of
  $C_0(\partial_RE)$. So we just have to show that a regular Borel probability
  measure $\mu$ on $\partial_RE$ satisfies that $\mu(\tpa_e(A))=M(e)\mu(A)$
  for every $e\in E^1$ and every Borel measurable subset $A$ of
  $\tdomain{e}$ if and only if  $\int f\circ\tpa_{e}^{-1}\ d\mu=M(e)\int f\ d\mu$ for
  every $e\in E^1$ and every $f\in C_0\bigl(\tdomain{e}\bigr)$.

  For each $e\in E^1$ let $L^1(\tdomain{e})$ denote the set of
  functions on $\tdomain{e}$ which are integrable with respect to the
  restriction of $\mu$ to $\tdomain{e}$, and let $||\cdot||_1$ be the
  subnorm given by
  \begin{equation*}
    ||f||_1=\int_{\tdomain{e}}|f|d\mu.
  \end{equation*}
  We then have that
  $C_0(\tdomain{e})$ is dense in $L^1(\tdomain{e})$ with respect to
  $||\cdot ||_1$. It follows that if the identity $\int f\circ\tpa_{e}^{-1}\
  d\mu=M(e)\int f\ d\mu$ holds for every $f\in
  C_0\bigl(\tdomain{e}\bigr)$, then it holds for every $f\in
  L^1\bigl(\tdomain{e}\bigr)$. Then in particular
  \begin{equation*}
    \mu(\tpa_e(A))=\int \cha{\tpa_e(A)}\ d\mu=\int
    \cha{A}\circ\tpa_{e\inv}\ d\mu=M(e)\int\cha{A}\ d\mu=M(e)\mu(A)
  \end{equation*}
  for every Borel measurable subset $A$ of $\tdomain{e}$.

  If, on the other hand, $\mu(\tpa_e(A))=M(e)\mu(A)$
  for every Borel measurable subset $A$ of
  $\tdomain{e}$, then the identity $\int f\circ\tpa_{e}^{-1} \ d\mu=M(e)\int f\
  d\mu$ holds for every $f\in
  L^1\bigl(\tdomain{e}\bigr)$ and in particular for every $f\in
  C_0\bigl(\tdomain{e}\bigr)$.
\end{proof}

\begin{proposition} \label{prop:int}
	In the setting of Theorem \ref{thm:kms}, the map
	\begin{equation*}
    \mu\mapsto \left(f\mapsto \int f\ d\mu\right)
  \end{equation*}
 is a convex isomorphism from $C^\beta$ to $B^\beta$ for $\beta\in [0,\infty)$, and a convex isomorphism from $C^{\operatorname{gr}}$ to $B^{\operatorname{gr}}$.
\end{proposition}

\begin{proof} Apply Lemma~\ref{lem:riesz} with the function $M:E^1\to
[0,\infty[$ given by $M(e)=(N(e))^{-\beta}$ when $\beta<\infty$, and $M(e)=0$ in case of  ground states.
\end{proof}

\begin{lemma} \label{lemma:hovedto}
  Let $E$ be a directed graph, let $R$ be a subset of $E^0_\reg$, and let $M$ be a function from $E^1$ to $[0,1]$. Then
  \begin{equation}
    \omega\mapsto \bigl(v\mapsto \omega(\cha{\cyl{v}\cap\partial_R
    E})\bigr) \label{eq:3}
  \end{equation}
  is a bijective correspondence between the set of states $\omega$ of
  $C_0(\partial_R E)$ such that
  $\omega(f\circ\phi_e\inv)=M(e)\omega(f)$ for all $e\in E^1$ and all $f\in C_0\bigl(\domain{e}\bigr)$, and the set of
  functions $m:E^0\to [0,1]$ satisfying \begin{enumerate}\renewcommand{\theenumi}{m\arabic{enumi}'}
  	\item\label{item:n1} $\sum_{v\in E^0}m(v)=1$;
    \item\label{item:n2} $m(v)=\sum_{e\in vE^1}M(e)m\bigl(r(e)\bigr)$ if $v\in R$;
    \item\label{item:n3} $m(v)\ge\sum_{e\in F}M(e)m\bigl(r(e)\bigr)$ for every finite subset $F$ of $E^1$.
  \end{enumerate}
\end{lemma}

\begin{proof}
  Let $\omega$ be a state of $C_0(\partial_R E)$ such that
  $\omega(f\circ\phi_e\inv)=M(e)\psi(f)$ for all $e\in E^1$ and $f\in C_0(\domain{e})$. Let $m$ be the function from $E^0$ to $[0,1]$
  given by
  \begin{equation*}
  	m(v)=\omega(\cha{\cyl{v}\cap\partial_R E}).
  \end{equation*}
  Now, if $F$ runs over the finite subsets of $E^0$, then $\{\sum_{v\in
    F}\cha{\cyl{v}\cap\partial_R E}\}_{F}$ is an
  approximate unit for $C_0(\partial_R E)$. Hence $m$ satisfies \eqref{item:n1}.

  To show \eqref{item:n2} and \eqref{item:n3} notice first that if $e\in E^1$,
  then
  \begin{align*}
    \omega\bigl(\cha{\rang{e}}\bigr)
    &=
    \omega\bigl(\cha{\domain{e}}\circ\phi_e\inv\bigr) =
    M(e) \omega\bigl(\cha{\domain{e}}\bigr)\\
    &=
    M(e) \omega\bigl(\cha{\cyl{r(e)}\cap\partial_R E}\bigr) =
    M(e)m\bigl(r(e)\bigr).
  \end{align*}
  If $v\in R$, then
  $\cha{\cyl{v}\cap\partial_R E}=\sum_{e\in vE^1}\cha{\rang{e}}$ by \eqref{eq:Zv}. Hence
  \begin{equation*}
      m(v)=\omega\bigl(\cha{\cyl{v}\cap\partial_R E}\bigr)
      =\smashoperator{\sum_{e\in vE^1}}\omega\bigl(\cha{\rang{e}}\bigr)
      =\smashoperator{\sum_{e\in vE^1}}M(e)m\bigl(r(e)\bigr),
  \end{equation*}
which gives \eqref{item:n2}.  If $v\in E^0$ and $F$ is a finite subset of $vE^1$, then
  $\cha{\cyl{v}\cap\partial_R E}\ge\sum_{e\in F}\cha{\rang{e}}$, so \eqref{item:n3} follows from the calculations
  \begin{equation*}
      m(v)=\omega\bigl(\cha{\cyl{v}\cap\partial_R E}\bigr)
      \ge\smashoperator{\sum_{e\in F}}\omega\bigl(\cha{\rang{e}}\bigr)
      =\smashoperator{\sum_{e\in F}}M(e)m\bigl(r(e)\bigr).
  \end{equation*}

  Since $\omega(\cha{\cyl{eu}\cap\partial_R E})=\omega(\cha{\cyl{u}\cap\partial_R E}\circ\phi_e\inv)=
  M(e)\omega(\cha{\cyl{u}\cap\partial_R E})$ for all
  $e\in E^1$ and all $u\in E^*$ with $s(u)=r(e)$, the restriction
  of $\omega$ to $\{\cha{\cyl{u}\cap\partial_R E}\mid u\in E^*\}$ is completely
  determined by the restriction
  of $\omega$ to $\{\cha{\cyl{v}\cap\partial_R E}\mid v\in E^0\}$. As seen in the proof of Lemma~\ref{lem:span-partialcp}, the
space  $\spa\{\cha{\cyl{u}\cap\partial_R E}\mid u\in E^*\}$ is dense in $C_0(\partial_R E)$. Therefore the correspondence
  given in \eqref{eq:3} is injective.

  We will now prove that it is surjective. Let $m:E^0\to [0,1]$ be a
  function that satisfies \eqref{item:n1}-\eqref{item:n3}.
  For each $u=e_1e_2\dotsm e_k\in E^*$, set
  \begin{equation*}
  \tilde{m}(u)=  M(e_1)M(e_2)\dotsm M(e_k)m(r(e_k)).
  \end{equation*}
  Straightforward calculations show that for $u\in E^*$,
  \begin{align*}
    \tilde{m}(u)&=\smashoperator{\sum_{e\in r(u)E^1}}\tilde{m}(ue),\text{ if }r(u)\in R,\\
    \tilde{m}(u)&\ge\smashoperator{\sum_{e\in F}}\tilde{m}(ue),
    \text{ if  $F$ is a finite subset of }r(u)E^1.
  \end{align*}

	Since $\{\cha{\cyl{u}}\mid u\in E^*\}$ is a linearly independent subset of $C_0(E^{\le\infty})$, it follows that there exists a linear map $\tilde\omega_m$ from $\spa\{\cha{\cyl{u}}\mid u\in E^*\}$ to $\C$ which maps $\cha{\cyl{u}}$ to $\tilde{m}(u)$ for $u\in E^*$. We show next that $\tilde\omega_m$ extends to a state of $C_0(E^{\le\infty})$. To begin with, we show that $\tilde\omega_m$ is bounded and its norm is not greater than 1. Let $f\in \spa\{\cha{\cyl{u}}\mid u\in E^*\}$. Then there exist a finite subset $F$ of $E^*$ and complex numbers $(c_u)_{u\in F}$ such that
  \begin{equation*}
  f=\sum_{u\in F}c_u \cha{\cyl{u}},
 \end{equation*}
	and such that $(u\in F \text{ and }u'\le u) \Rightarrow u'\in F$. We then have that
	\begin{equation*}
		\tilde\omega_m(f)=\sum_{u\in F}c_u\tilde{m}(u)=\sum_{u\in F}\biggl(\sum_{u'\le u}c_{u'}\biggr)\biggl(\tilde{m}(u)-\sum_{e\in r(u)E^1,\ ue\in F}\tilde{m}(ue)\biggr).
	\end{equation*}
	Since $f(u)=\sum_{u'\le u}c_{u'}$ for $u\in F$ and $$\sum_{u\in F}\biggl(\tilde{m}(u)-\sum_{e\in r(u)E^1,\ ue\in F}\tilde{m}(ue)\biggr)=\sum_{v\in E^0\cap F}{m}(v)\le 1,$$ it follows from H\"older's inequality that $\abs{\tilde{\omega}_m(f)}\le \norm{f}_\infty$. Thus we can extend $\tilde\omega_m$ to a bounded linear functional with norm less than or equal to 1 on $\spac\{\cha{\cyl{u}}\mid u\in E^*\}=C_0(E^{\le\infty})$. The family $(\sum_{v\in F}\cha{\cyl{v}})_{F}$ indexed over finite subsets $F$ of $E^0$ forms an approximate unit for $C_0(E^{\le\infty})$, and \eqref{item:m1} therefore implies that $\lim_{F\subset E^0}\tilde\omega_m( \sum_{v\in F}\cha{\cyl{v}})=1$. Thus $\tilde\omega_m$ is a state of $C_0(E^{\le\infty})$ (e.g. from \cite[Theorem 3.3.3]{MR1074574}).
	
	It follows from Proposition \ref{prop:et}(v) and the definition of $\partial_RE$ that $\{f\in C_0(E^{\le\infty})\mid f(x)=0\text{ for all }x\in\partial_RE\}=\spac\{\cha{\{u\}}\mid r(u)\in R\}$. Since for
$r(u)\in R$ we have
	\begin{equation*}
		\tilde\omega_m(\cha{\{u\}})=\tilde\omega_m\biggl(\cha{\cyl{u}}-\sum_{e\in r(u)E^1}\cha{\cyl{ue}}\biggr)=\tilde{m}(u)-\smashoperator{\sum_{e\in r(u)E^1}}\tilde{m}(ue)=0,
	\end{equation*}
it follows that $\tilde\omega_m$ induces a state $\omega_m$ on $C_0(\partial_RE)$ which maps $\cha{\cyl{u}\cap\partial_RE}$ to $\tilde{m}(u)$ for $u\in E^*$.

  Let $e\in E^1$ and $u\in r(e)E^*$. Then
  \begin{equation*}
    \omega_m(\cha{\cyl{u}\cap\partial_RE}\circ\phi_e\inv)=\omega_m(\cha{\cyl{eu}\cap\partial_RE})=\tilde{m}(eu)
    =M(e)\omega_m(\cha{\cyl{u}\cap \partial_RE}).
  \end{equation*}
  As already noticed, $\spac\{\cha{\cyl{u}\cap\partial_R E}\mid u\in r(e)E^*\}=C_0(U(e\inv))$, and therefore
  $\omega_m(f\circ\phi_e\inv)=M(e)\omega_m(f)$ for every $f\in C_0(\domain{e})$. Since
  $\omega_m(\cha{\cyl{v}\cap\partial_R E})=m(v)$ for every $v\in E^0$, we have shown the claimed surjectivity.
\end{proof}

\begin{proposition} \label{prop:function}
	In the setting of Theorem \ref{thm:kms}, the map from \eqref{eq:3}
 is a convex isomorphism from $B^\beta$ to $D^\beta$ for $\beta\in[0,\infty)$, and from
$B^{\operatorname{gr}}$ to $D^{\operatorname{gr}}$.
\end{proposition}

\begin{proof} Apply Lemma~\ref{lemma:hovedto} with the function $M:E^1\to
[0,\infty)$ given by $M(e)=(N(e))^{-\beta}$ when $\beta<\infty$, and $M(e)=0$ in case of ground states.
\end{proof}

\section{Extremal KMS states}\label{section:extreme}

In this section we aim to give a description of the  extreme points of $D^\beta$ for $\beta\geq 0$. Ideally, we want a description that is valid for arbitrary graphs. However, this task seems to be quite difficult. We will identify certain subsets of the set of extreme points of $D^\beta$, see Theorem~\ref{thm:finitetype-reg} and Theorem~\ref{thm:infinitetype-crit}. The strategy will be to  describe the supports of the corresponding measures in $C^\beta$. For certain families of graphs (in particular all graphs with finitely many vertices), our description will give all the extremal KMS states.

Throughout this section $E$ will denote a directed graph, $R$ a subset of $E^0_\reg$, $N:E^1\to (1,\infty)$ a function, and $\beta\in [0,\infty)$. We extend the function $N$ to $E^*$ by letting $N(v)=1$ for $v\in E^0$ and by letting $N(u)=N(u_1)\dotsm N(u_n)$ for $u=u_1\dotsm u_n\in E^n$ and $n\geq 1$. We adopt the following convention: if $m\in D^\beta$ and
$\mu$ is the unique element of $C^\beta$ given by Propositions~\ref{prop:int} and \ref{prop:function} such that
\begin{equation}\label{eq:convention-mu-m}
\mu({Z(v')}\cap \partial_R E)=m(v')
\end{equation}
 for all $v'\in E^0$,  we say that $\mu$ is the \emph{measure associated to $m$}.

To begin with, we divide the elements of $D^\beta$  in terms of finite and infinite type  measures in $C^\beta$, similar to what is done in \cite{MR1953065}.

\begin{definition} Let $m\in D^\beta$ and let $\mu$ be the measure associated to $m$. Then
\begin{enumerate}
\item $m$ is of \emph{finite type} if  $\mu(E^*\cap \partial_R E)=1$, and
\item $m$ is  of \emph{infinite type} if  $\mu(E^\infty)=1$.
\end{enumerate}
We let $\Cf$ and $D^{\beta}_{\operatorname{inf}}$  denote, respectively, the sets of $m$ of finite type and of infinite type.
\end{definition}

For the infinite type measures we introduce the following refinement.

\begin{definition} Let $E$ be a directed graph.

 \textnormal{(a)} We define the set $E^\infty_{\operatorname{rec}}$ of \emph{recurrent} paths to be the collection of all infinite paths that meet some vertex of $E^0$ infinitely many times: thus $x\in E^\infty_{\operatorname{rec}}$ if and only if there is $v\in E^0$ such that $\{u\in E^*\mid u<x, r(u)=v\}$ is infinite.

\textnormal{(b)} We define the set $E^\infty_{\operatorname{wan}}$ of \emph{wandering} paths to be the collection of all $x\in E^\infty$ such that for every $v\in E^0$, the set $\{u\in E^*\mid u<x, r(u)=v\}$ is finite.

Note that $\Ewan=\emptyset$ when $E^0$ is finite. In general, $E^\infty=\Erec\sqcup \Ewan$.
\end{definition}

\begin{definition} Let $\mu\in C^\beta$. Following \cite{Tho2}, we say that
\begin{enumerate}
\item $\mu$ is   \emph{conservative} if it has support on $E^\infty_{\operatorname{rec}}$, and
\item $\mu$ is \emph{dissipative} if  it has support on $E^\infty_{\operatorname{wan}}$.
\end{enumerate}
We let  $\Cinfa$ and $\Cinfb$ denote, respectively, the sets of functions $m$ whose associated measure via \eqref{eq:convention-mu-m} is  conservative, respectively dissipative. In either instance we shall refer to $m$ itself as being conservative or dissipative.
\end{definition}

Note that $\Cinfb=\emptyset$ if $\Ewan=\emptyset$ (in particular if $E^0$ is finite). Example~\ref{ex:dis} and Example~\ref{ex:motivating} provide examples where $\Cinfb\ne\emptyset$.

\begin{remark}\label{rmk:decom}
	The three subsets $E^*\cap \partial_R E$, $E^\infty_{\operatorname{rec}}$, and $E^\infty_{\operatorname{wan}}$ of $\partial_RE$ are all invariant under the partial action $\Phi$. It follows that every $m\in D^\beta$ in a unique way can be written as a convex combination of an element of $\Cf$, an element of $\Cinfa$ and an element of $\Cinfb$.
\end{remark}

It follows from Remark \ref{rmk:decom} that the set of extreme points of $D^\beta$ is the disjoint union of the sets of extreme points of $\Cf$, $\Cinfa$, and $\Cinfb$. We will in Theorem~\ref{thm:finitetype-reg} and Theorem~\ref{thm:infinitetype-crit} identify the extreme points of $\Cf$ and $\Cinfa$. Hence, if $\Cinfb=\emptyset$ (in particular if $E^0$ is finite), then we obtain a complete description of all the extreme points of $D^\beta$ and thus a complete description of all the KMS states of $(C^*(E,R),\sigma)$.

In order to define distinguished sets of vertices on which some of the extreme points of $D^\beta$ will be
supported we need to introduce some notation. For $v\in E^0$, let
$$vE^*v=\{u\in E^*\mid s(u)=r(u)=v\}$$
be the collection of all finite paths starting and ending at $v$ (also referred to as \emph{loops} or \emph{cycles} at $v$), and let
\begin{equation*}
	\Evav=\{u\in E^*\mid s(u)=r(u)=v,\ u\ne v,\ r(u')\ne v\text{ for any }v<u'<u\}
\end{equation*}
be the set of paths starting and ending at $v$ with length at least 1 and containing no proper subpath that is a loop at $v$ (these are sometime called \emph{simple} loops or cycles). Notice that $\Evav$ might be empty, but that $v\in vE^*v$. In fact,
\begin{equation*}
	vE^*v=\{v\}\cup\bigcup_{n=1}^\infty\{u_1u_2\cdots u_n\mid u_1,u_2,\dots,u_n\in\Evav\}.
\end{equation*}

Recall that $E^*v=\{u\in E^*\mid r(u)=v\}$ is the set of finite paths ending in $v$. We let
\begin{equation*}
\Eva=\{u\in E^*\mid r(u)=v,\ r(u')\ne v\text{ for any }u'<u\}
\end{equation*}
be the set of finite paths ending in $v$ such that no proper subpath has range $v$. Notice that $E^*v$ and $\Eva$ are both non-empty since $v\in E^*_av\subseteq E^*v$.

Next we associate partition functions to the sets $\Evav$ and $\Eva$ as follows:
\begin{align}
	\Zvav(\beta)&=\sum_{u\in\Evav}N(u)^{-\beta}\label{eq:Zvav}\\
	\Zva(\beta)&=\sum_{u\in\Eva}N(u)^{-\beta}.\label{eq:Zva}
\end{align}
Notice that $\Zvav(\beta)$ might be 0 (since $\Evav$ might be empty), whereas $\Zva(\beta)\ge 1$ (because $v\in\Eva$).

We now define the following distinguished sets of vertices.
\begin{gather}
\Ereg=\{v\in E^0\mid \Zva(\beta)<\infty \text{ and }\Zvav(\beta)< 1\}, \label{eq:Ereg}\\
	\Ecrit=\{v\in E^0\mid \Zva(\beta)<\infty\text{ and }\Zvav(\beta)= 1\}.\label{eq:Ecrit}
\end{gather}
The abbreviations in the notation stand for regular and critical, respectively. We shall refer to $\Eequ:=\{v\in E^0\mid \Zva(\beta)<\infty\text{ and }\Zvav(\beta)\le 1\}$ as the set of equivariant points. The main results of this section will establish that elements in $\Cf$ are determined by $\Ereg\setminus R$, and elements in $\Cinfa$ by (equivalence classes of elements in) $\Ecrit$, cf. Theorems~\ref{thm:finitetype-reg} and \ref{thm:infinitetype-crit}. In particular $\Cf=\emptyset$ if and only if $\Ereg\setminus R=\emptyset$, and $\Cinfa=\emptyset$ if and only if $\Ecrit=\emptyset$.

Towards defining extreme points of $D^\beta$ we need to keep track of paths between a pair of vertices. Thus, for $v,v'\in E^0$ we let
\begin{equation*}
	v'E^*v=\{u\in E^*\mid s(u)=v',\ r(u)=v\}
\end{equation*}
be the set of finite paths starting at $v'$ and ending at $v$, and we let
\begin{equation*}
	\Es{v'}{v}=\{u\in E^*\mid s(u)=v',\ r(u)=v,\ r(u')\ne v\text{ for any }u'<u\}
\end{equation*}
be the set of finite paths starting at $v'$ and ending at $v$ such that no proper subpath has range $v$. In general, the sets $v'E^*v$ and $\Es{v'}{v}$ could be empty. Note however that $\Es{v'}{v}\subseteq v'E^*v$ and that $\Es{v}{v}=\{v\}$.
%In fact, $\Es{v}{v}=\{v\}\cup\Evav$.
Notice also that $v'E^*v=\{uu'\mid u\in\Es{v'}{v},\ u'\in{v}E^*{v}\}$ and $\Eva=\bigcup_{v'\in E^0}\Es{v'}{v}$.

\begin{definition}\label{def:m-of-v}
For $v\in\Eequ$, let $m^\beta_v:E^0\to [0,\infty]$ be given by
	\begin{equation}\label{eqdef:m_v}
		m^\beta_v(v')=\sum_{u\in\Es{v'}{v}}N(u)^{-\beta}(\Zva(\beta))^{-1}.
	\end{equation}
\end{definition}

We are now in a position to state the first main result of this section, which provides a description of the elements of $\Cf$.

\begin{theorem}\label{thm:finitetype-reg}
Let $\beta\in [0,\infty)$. The map $W_{\operatorname{fin}}$ from $\Cf$ to the set of $[0,1]$-valued functions on $\Ereg\setminus R$ given by
\begin{equation}\label{eqdef:W-fin}
	W_{\operatorname{fin}}(m)(v)=\frac{\Zva(\beta)}{1-\Zvav(\beta)} \Biggl(m(v)-\sum_{e\in vE^1} N(e)^{-\beta}m(r(e))\Biggr)
\end{equation}
for $v\in\Ereg\setminus R$, is a convex isomorphism
%from $\Cf$ to the set
onto $\bigl\{\epsilon:\Ereg\setminus R\to [0,1]\biggm\vert \sum_{v\in\Ereg \setminus R}\epsilon(v)=1\bigr\}$.
	The inverse of $W_{\operatorname{fin}}$ is the map $\epsilon\mapsto \sum_{v\in\Ereg\setminus R}\epsilon(v)m^\beta_v$.
\end{theorem}

The proof of this theorem will require some preparation in the form of a series of preliminary results.

It will be convenient to have a notation for the function on $E^0$ appearing in the right-hand side of \eqref{eqdef:W-fin}. Therefore, for $m:E^0\to [0,1]$ satisfying \eqref{item:m3}, we let $S(m)$ be the function from $E^0$ to $[0,1]$ given by
\begin{equation*}
	S(m)(v)= m(v)-\sum_{e\in vE^1} N(e)^{-\beta}m(r(e)), \text{ for }v\in E^0.
\end{equation*}
Notice that $m$ satisfies \eqref{item:m2} if and only if $S(m)(v)=0$ for all $v\in R$, and that $m$ is an eigenvector with eigenvalue 1 of the matrix $(\sum_{e\in v'E^1v}N(e))_{v',v\in E^0}$ if and only if $S(m)(v)=0$ for all $v\in E^0$.

Some properties of this function $S$ are collected in the following lemma.

\begin{lemma}\label{lem:aboutS}
Let $m\in D^\beta$ and let $\mu\in C^\beta$ be the measure associated to $m$.

\textnormal{(a)} We have $\mu(\{v\})=S(m)(v)$ for all $v\in E^0\setminus R$.

\textnormal{(b)} We have $S(m)(v)=0$ for all $v\in E^0$ if and only if $m\in \Cinf$.
\end{lemma}

\begin{proof} The regularity of $\mu$ implies that
	\begin{align*}
		\mu(\{v\})&=\mu\biggl(\cyl{v}\setminus\bigcup_{e\in vE^1}\cyl{e}\biggr)\\
		&=\mu(\cyl{v}\cap\partial_RE)-\sum_{e\in vE^1}\mu(\cyl{e}\cap\partial_RE)\\
		&=m(v)-\sum_{e\in vE^1}N(e)^{-\beta}m(r(e))\\
		&=S(m)(v),
	\end{align*}
as claimed in (a).

To prove (b), assume first that $S(m)(v)=0$ for all $v\in E^0$. By (a), $\mu(\{v\})=0$ for all $v\in E^0\setminus R$. Hence by the scaling condition in $C^\beta$, $\mu(\{u\})=N(u)^{-\beta}\mu(\{r(u)\})=0$ for all $u\in E^*$ with $r(u)\notin R$.  Thus,
$$
\mu(E^\infty)=\mu(\partial_R E\setminus \{u\in E^*\mid r(u)\notin R\})=\mu(\partial_R E)=1.
$$
Conversely, if $m\in \Cinf$, then $S(m)(v)=\mu(\{v\})=0$ for $v\in E^0\setminus R$ by (a), and $S(m)(v)=0$ for $v\in R$ since $m$ satisfies \eqref{item:m2}.
\end{proof}

It follows from Lemma~\ref{lem:aboutS} (b) that $D^\beta_{inf}$ is
the set of normalized eigenvectors with eigenvalue 1 of the matrix
$(\sum_{e\in v' E^1 v} N(e))_{v',v\in E^0}$ (and to the normalized
eigenvectors with eigenvalue $\exp(\beta)$ of the adjacency matrix of
$E$ if $N(e)=\exp(1)$ for all $e\in E^1$).

 For $v\in E^0$ define a partition function
\begin{equation}\label{eq:Zv}
Z_v(\beta)=\sum_{u\in E^*v}N(u)^{-\beta}.
\end{equation}
Similar to the terminology used in \cite{MR1953065} we call $Z_v(\beta)$ the partition function with fixed-target $v$. Clearly  $\beta_1\le\beta_2$ implies $Z_v(\beta_2)\le Z_v(\beta_1)$. Thus, if $Z_v(\beta)$ is convergent, then $Z_v(\beta')$ is convergent for all $\beta'\ge \beta$.

It will be useful to know that the map $(u_0,u_1,\dots,u_n)\mapsto u_0u_1\dotsm u_n$ is a bijection
\begin{equation}\label{eq:bijection1}
\Eva\times\bigcup_{n=0}^\infty(\Evav)^n\to E^*v,
\end{equation}
where $(\Evav)^0=\{v\}$.

%\label{eq:Eequ}
\begin{proposition} \label{prop:E} Let $\beta\in [0,\infty)$. The following hold:
	\mbox{ }
	\begin{enumerate}
		\item \label{item:a1} $Z_v(\beta)=\Zva(\beta)\Bigl(1+\sum_{n=1}^\infty(\Zvav(\beta))^n\Bigr)$ for any $v\in E^0$.
		\item \label{item:a2} $\Ereg=\{v\in E^0\mid Z_v(\beta)<\infty\}$.
		%\item \label{item:a3} $\Eequ$ is the disjoint union of $\Ereg$ and $\Ecrit$.
		\item \label{item:a3.5} Let $v\in E^0$ and $m\in D^\beta$. If $m(v)\ne 0$, then $\Zva(\beta)\le \frac{1}{m(v)}$.
		\item \label{item:a4} Let $v\in E^0$. If there exists an $m\in D^\beta$ such that $m(v)\ne 0$, then $v\in\Eequ$.
	\end{enumerate}
\end{proposition}

\begin{proof}
Assertion \eqref{item:a1} follows directly from  \eqref{eq:bijection1}, and assertion \eqref{item:a2} follows directly from \eqref{item:a1}.
	
We next prove	\eqref{item:a3.5}. Suppose that $m(v)\ne 0$. Since $m\in D^\beta$, Proposition \ref{prop:function} gives a unique $\omega\in B^\beta$ such that $\omega(\cha{\cyl{v'}\cap\partial_RE})=m(v')$ for all $v'\in E^0$. We let $\psi$ be the element in $A^\beta$ corresponding to $\omega$ under the isomorphism of Proposition~\ref{prop:state}. If $u_1,u_2\in\Eva$ and $u_1\ne u_2$, then $\cyl{u_1}\cap\cyl{u_2}=\emptyset$. We claim that
\begin{equation}\label{eq:sum-less-one}
\sum_{u\in\Eva} \omega(\cha{\cyl{u}\cap\partial_R E})\leq 1.
\end{equation}
To see this, use that $\sqcup_{u\in \Eva}Z(u)\subseteq \sqcup_{v'\in J} Z(v')$, where $J=\{s(u)\mid u\in \Eva\}$, to bound the left hand side of \eqref{eq:sum-less-one} by $\sum_{v'\in J} \overline{\psi}(\cha{Z(v')})$, with $\overline{\psi}$ denoting the state extension of $\psi\vert_{C_0(\delta_R E)}$ to $C_0(E^{\leq \infty})$. The fact that the
net $(\sum_{v''\in F}\cha{Z(v'')})$ indexed over finite subsets of $E^0$ forms an approximate unit for $C_0(E^{\leq \infty})$ then gives \eqref{eq:sum-less-one}. The scaling condition in $B^\beta$ therefore implies that
$$\sum_{u\in\Eva}N(u)^{-\beta}m(v)\le 1,$$
and thus $\Zva(\beta)=\sum_{u\in\Eva}N(u)^{-\beta}\le \frac{1}{m(v)}$.

Finally, to prove \eqref{item:a4}, assume that $m\in D^\beta$ and $m(v)\ne 0$. Let $\omega\in B^\beta$ be as above. If $u_1,u_2\in\Evav$ and $u_1\ne u_2$, then $Z(u_1)\cap Z(u_2)=\emptyset$, hence $\sqcup_{u\in\Evav}Z(u)\subseteq Z(v)$. It follows from the scaling condition in $B^\beta$ that
	$$\sum_{u\in\Evav}N(u)^{-\beta}\omega(\cha{\cyl{v}\cap\partial_RE})\le\omega(\cha{\cyl{v}\cap\partial_RE}).$$
	Thus, since $\omega(\cha{\cyl{v}\cap\partial_RE})=m(v)\ne 0$, we get that $\Zvav(\beta)=\sum_{u\in\Evav}N(u)^{-\beta}\le 1$. Since $\Zva(\beta)<\infty$ by \eqref{item:a3.5}, we have that $v\in\Eequ$.
\end{proof}

\begin{lemma}\label{lem:mineq}
	Let $m\in D^\beta$ and $v_1,v_2\in E^0$. Then $m(v_2)\ge \sum_{u\in\Es{v_2}{v_1}}N(u)^{-\beta}m(v_1)$.
\end{lemma}

\begin{proof}
If $\Es{v_2}{v_1}=\emptyset$ there is nothing to prove. Assume $\Es{v_2}{v_1}\neq\emptyset$.	Let $\mu$ be the measure associated to
$m$ as given by \eqref{eq:convention-mu-m}. The scaling condition in $C^\beta$ implies that
$$
\mu(\{ux\mid x\in\cyl{r(u)}\}\cap\partial_RE)=N(u)^{-\beta}\mu(Z(r(u))\cap\partial_R E)=N(u)^{-\beta}m(r(u))
$$
 for any $u\in E^*$.
	
	If $u_1,u_2\in\Es{v_2}{v_1}$ and $u_1\ne u_2$, then $\{u_1x\mid x\in\cyl{v_1}\}$ and $\{u_2x\mid x\in\cyl{v_1}\}$ are two disjoint subsets of $\cyl{v_2}$. Hence
	\begin{align*}
		m(v_2)&=\mu(\cyl{v_2}\cap\partial_RE)\ge \sum_{u\in\Es{v_2}{v_1}}\mu(\{ux\mid x\in\cyl{v_1}\}\cap\partial_RE)\\ &=\sum_{u\in\Es{v_2}{v_1}}N(u)^{-\beta}m(v_1).
	\end{align*}
\end{proof}

For later use, we record the following fact.

\begin{lemma}\label{lem:moreaboutS}
Let $\beta\in [0,\infty)$. Then $S(m)(v)\le m(v)(1-\Zvav(\beta))$ for any $m\in D^\beta$ and $v\in \Eequ$. In particular,
$S(m)(v)=0$ for $v\in \Ecrit$.
\end{lemma}

\begin{proof} Let $m\in D^\beta$ and $v\in \Eequ$. Applying Lemma~\ref{lem:mineq} with $v_1=v$ and $v_2=r(e)$ for all $e\in vE^1$ gives that
$$
\sum_{e\in vE^1}N(e)^{-\beta}m(r(e))\geq \sum_{e\in vE^1} N(e)^{-\beta}\Biggl(\sum_{u\in \Es{r(e)}{v}}N(u)^{-\beta}m(v)\Biggr).
$$
Since $(e,u)\to ue$ implements a bijection between $\{(e,u)\mid e\in vE^1\times,\ u\in \Es{r(e)}{v}\}$ and $\Evav$, this inequality shows that
$\sum_{e\in vE^1}N(e)^{-\beta}m(r(e))\geq m(v)\sum_{u'\in \Evav}N(u')^{-\beta}=m(v)\Zvav(\beta)$. The first claim thus follows, and it implies the second claim by \eqref{eq:Ecrit}.
\end{proof}

\begin{proposition} \label{prop:m}
	\textnormal{(a)} For each $v\in\Eequ$, the function $m_v^\beta$ satisfies \eqref{item:m1} and \eqref{item:m3}.  Further,
 $$
S(m^\beta_v)(v')=\begin{cases}\frac{1-\Zvav(\beta)}{\Zva(\beta)} &\text{ if }v'=v\\ 0&\text{ if }v'\ne v.
\end{cases}$$

\textnormal{(b)}  $m^\beta_v\in D^\beta$ if and only if $v\in\Ecrit$ or $v\in\Ereg\setminus R$.

\textnormal{(c)}  $m^\beta_v\in\Cf$ if and only if $v\in\Ereg\setminus R$.

\textnormal{(d)}  $m^\beta_v\in\Cinf$ if and only if $v\in\Ecrit$.

\end{proposition}

Note that when $R=\emptyset$, i.e. we are looking at the Toeplitz algebra of the graph, (b) shows that every element $v\in \Eequ$ defines a function
$m_v^\beta$ in $D^\beta$.

\begin{proof} Since $\Es{v}{v}=\{v\}$, we have $m^\beta_v(v)=\sum_{u\in\Es{v}{v}}N(u)^{-\beta}(\Zva(\beta))^{-1}=(\Zva(\beta))^{-1}$. Using the decomposition $\Evav=\bigsqcup_{e\in vE^1} \Es{r(e)}{v}$, it follows that
	\begin{align}
		\sum_{e\in vE^1}N(e)^{-\beta}m_v^\beta(r(e))&=\sum_{e\in vE^1}N(e)^{-\beta}\sum_{u\in\Es{r(e)}{v}}N(u)^{-\beta}(\Zva(\beta))^{-1}\notag\\
		&=\sum_{u'\in\Evav}N(u')^{-\beta}(\Zva(\beta))^{-1}.\label{eq:prove-m1-m3}
	\end{align}
Thus $\sum_{e\in vE^1}N(e)^{-\beta}m_v^\beta(r(e))={\Zvav}(\beta)m^\beta_v(v)$. Now the assumption that $v\in \Eequ$ implies that
$m_v^\beta(v)$ satisfies \eqref{item:m3}. By reorganizing terms we obtain $S(m^\beta_v)(v)=\frac{1-\Zvav(\beta)}{\Zva(\beta)}$. If $v'\ne v$, then $S(m^\beta_v)(v')=0$ follows from the calculations
	\begin{align*}
		\sum_{e\in v'E^1}N(e)^{-\beta}m_v^\beta(r(e))&=\sum_{e\in v'E^1}N(e)^{-\beta}\sum_{u\in\Es{r(e)}{v}}N(u)^{-\beta}(\Zva(\beta))^{-1}\\
		&=\sum_{u'\in\Es{v'}{v}}N(u')^{-\beta}(\Zva(\beta))^{-1}\\
&=m^\beta_v(v').
	\end{align*}
To finish the proof of (a) it remains to show that $m_v^\beta$ satisfies  \eqref{item:m1}. This follows from the decomposition $
\Eva=\bigsqcup_{v'\in E^0}\Es{v'}{v}$ and  the calculations
	\begin{align*}
		\sum_{v'\in E^0}m_v^\beta(v')&=\sum_{v'\in E^0}\sum_{u\in\Es{v'}{v}}N(u)^{-\beta}(\Zva(\beta))^{-1}\\
&=\sum_{u\in\Eva}N(u)^{-\beta}(\Zva(\beta))^{-1}=1.
	\end{align*}
In particular, we have that $m^\beta_v(v')\in [0,1]$.

For (b), note that (a) implies that $m^\beta_v$ satisfies \eqref{item:m2} if and only if $v\in\Ecrit$ or $v\in\Ereg\setminus R$.

For (c) and (d), let  $\mu^\beta_v\in C^\beta$ be the measure associated to $m_v^\beta$ as in \eqref{eq:convention-mu-m}. Recall that the scaling condition in $C^\beta$ gives that
$$
\mu^\beta_v(\cyl{u}\cap\partial_RE)=N(u)^{-\beta}m^\beta_v(r(u))$$
for all $u\in E^*$.
	
Suppose that $v\in\Ereg\setminus R$. By Lemma~\ref{lem:aboutS} and part (a),
$$
\mu^\beta_v(\{v\})=S(m_v^\beta)(v)=\frac{1-\Zvav(\beta)}{\Zva(\beta)}.
$$
Using the above scaling condition, $\mu^\beta_v(\{u\})=N(u)^{-\beta}\frac{1-\Zvav(\beta)}{\Zva(\beta)}$ for any $u\in E^*v$. By \eqref{eq:bijection1},
	\begin{align*}
		\mu^\beta_v(E^*v)&=\sum_{u\in E^*v}\mu^\beta_v(\{u\})=\sum_{u\in E^*v}N(u)^{-\beta}\frac{1-\Zvav(\beta)}{\Zva(\beta)}\\
		 &=\Biggl(\sum_{u\in\Eva}N(u)^{-\beta}\Biggr)\Biggl(\sum_{n=0}^\infty(\Zvav(\beta))^n\Biggr)\frac{1-\Zvav(\beta)}{\Zva(\beta)}=1.
	\end{align*}
	Thus $\mu^\beta_v\in \Cf$ when $v\in\Ereg\setminus R$.
	
	Suppose next that $v\in\Ecrit$. Lemma~\ref{lem:aboutS}(a) and part (a) imply that
	\begin{equation*}
		\mu^\beta_v(\{v'\})=S(m^\beta_v)(v')=0
	\end{equation*}
	for any $v'\in E^0\setminus R$. Since clearly $S(m_v^\beta)(v')=0$ for all $v'\in R$, Lemma~\ref{lem:aboutS}(b) shows that  $m^\beta_v$ is of infinite type, i.e. belongs to $\Cinf$.
\end{proof}

The next step in our analysis is a more detailed study of the structure of the sets $\Eequ$, $\Ereg$ and $\Ecrit$, and the functions $m^\beta_v$. Given two vertices $v_1,v_2\in E^0$, we introduce the notation
\begin{align*}
v_1\succ v_2 &\text{ if }v_1E^*\cap E^*v_2\ne\emptyset, \text{ and }\\
v_1\sim v_2 &\text{ if }v_1\succ v_2 \text{ and }v_2\succ v_1.
\end{align*}

\begin{proposition}\label{prop:sums}
	\mbox{ }
	\begin{enumerate}
		\item \label{ item:gg1} If $v_1\in\Ereg$ and $v_2\succ v_1$, then $v_2\in\Ereg$.
		\item \label{ item:gg2} If $v_1\in\Eequ$ and $v_2\succ v_1$, then $v_2\in\Eequ$.
		\item \label{ item:gg3} If $v_1\in\Ecrit$ and $v_2\sim v_1$, then $v_2\in\Ecrit$.
		\item \label{ item:gg4} If $v_1\in\Eequ$, $v_2\succ v_1$, and $v_1\not\succ v_2$, then $v_2\in\Ereg$.
	\end{enumerate}
\end{proposition}

\begin{proof}
Assertion	\eqref{ item:gg1} follows from the fact that $Z_{v_1}(\beta)\geq N(u)^{-\beta}Z_{v_2}$ for every $u\in v_2E^*v_1$.
	
	For \eqref{ item:gg2}, notice that if $v_1\in\Eequ$ and $v_2\succ v_1$, then $m_{v_1}^\beta(v_2)\ne 0$. Assuming first that $R=\emptyset$, it follows from Proposition \ref{prop:E}\eqref{item:a4}  that $v_2\in\Eequ$. Since the definition of
the set $\Eequ$ does not depend on $R$, the claim is true in general.
	
To prove	\eqref{ item:gg3}, suppose $v_1\in\Ecrit$ and $v_1\sim v_2$. Then \eqref{ item:gg2}  implies that $v_2\in\Eequ=\Ereg\cup\Ecrit$, and it follows from \eqref{ item:gg1} applied to $v_1\succ v_2$ that if $v_2\in\Ereg$, then $v_1\in\Ereg$. Since the latter is not the case, we must have that $v_2\in\Ecrit$.
	
Finally, for	\eqref{ item:gg4} suppose $v_1\in\Eequ$, $v_2\succ v_1$, and $v_1\not\succ v_2$. It then  follows from \eqref{ item:gg2} that $v_2\in\Eequ=\Ereg\cup\Ecrit$. Assume for contradiction that $v_2\in\Ecrit$. Since $v_2\succ v_1$, we have $\Es{v_2}{v_1}\ne\emptyset$. Choose $u\in \Es{v_2}{v_1}$. Let $n\ge 1$. For $u_1,u_2,\dots,u_n\in v_2E_s^*v_2$ we have that $u_1u_2\dotsm u_nu\in E_a^*v_1$ (since $v_1\not\succ v_2$). Hence
	$$Z_{v_1}^a(\beta)\ge \sum_{n=1}^\infty \Bigl(Z_{v_2}^s(\beta)\Bigr)^nN(u)^{-\beta}=\infty,$$
	which contradicts the assumption that $v_1\in\Eequ$. Thus, $v_2\in\Ereg$.
\end{proof}

\begin{proof}[Proof of Theorem~\ref{thm:finitetype-reg}.] We must first prove that $W_{\operatorname{fin}}$ is well-defined. Let $m\in\Cf$ and let $\mu\in C^\beta$ be the measure associated to $m$. Then
\begin{align}
\sum_{v\in \Ereg\setminus R}W_{\operatorname{fin}}(m)(v)
&=\sum_{v\in \Ereg\setminus R}\frac{S(m)(v)\Zva(\beta)}{1-\Zvav(\beta)}\notag\\
&=\sum_{v\in \Ereg\setminus R}\frac{\mu(\{v\})\Zva(\beta)}{1-\Zvav(\beta)}\text{ by Lemma~\ref{lem:aboutS}(a)}\notag\\
&=\sum_{v\in \Ereg\setminus R} \mu(\{v\})Z_v(\beta) \text{ by Proposition~\ref{prop:E}}\notag\\
&=\sum_{v\in \Ereg\setminus R}\sum_{u\in E^*v}\mu(\{v\})N(u)^{-\beta}\notag\\
&=\sum_{v\in \Ereg\setminus R}\sum_{u\in E^*v}\mu(\{u\}).\label{eq:wf-well}
\end{align}

The scaling identity in $C^\beta$ and Lemma~\ref{lem:aboutS} imply that
\begin{equation}\label{eq:mu-u}
\mu(\{u\})=N(u)^{-\beta}\mu(\{r(u)\})=N(u)^{-\beta}S(m)(r(u))
\end{equation}
for all $u\in E^*\cap\partial_RE$.
We claim that $S(m)(r(u))=0$ unless $r(u)\in\Ereg\setminus R$. From
Lemma~\ref{lem:moreaboutS} and the definition of $S$ we have $S(m)(r(u))=0$ for $r(u)\in\Ecrit\cup R$. If $r(u)\notin \Eequ$, then Proposition~\ref{prop:E}(4) implies that $m(r(u))=0$, therefore also $S(m)(r(u))=0$. The claim follows and implies that
$\mu$ is supported on the finite paths that end in vertices $v\in \Ereg\setminus R$. Hence \eqref{eq:wf-well} gives that $\sum_{v\in \Ereg\setminus R}W_{\operatorname{fin}}(m)(v)=\mu(E^*\cap \partial_R E)=1$, which shows that $W_{\operatorname{fin}}$ is well-defined. Clearly $W_{\operatorname{fin}}$ is a convex map.

We claim next that  $\bigl\{\epsilon:\Ereg\setminus R\to [0,1]\bigm\vert \sum_{v\in\Ereg\setminus R}\epsilon(v)=1\bigr\} \subseteq W_{\operatorname{fin}}(\Cf)$, which gives surjectivity. Clearly for every $v\in \Ereg\setminus R$ the function $\delta_v: \Ereg\setminus R\to [0,1]$ defined by $\delta_v(v)=1$ and $\delta_v(v')=0$ when $v'\ne v$ belongs to the set on the left-hand side. The claim  follows because Proposition~\ref{prop:m} shows that  $W_{\operatorname{fin}}(m_v^\beta)=\delta_v$ and every $\epsilon$ can be written as $\epsilon=\sum_{v\in \Ereg\setminus R} \epsilon(v)\delta_v$.

Finally, assume that $W_{\operatorname{fin}}(m_1)=W_{\operatorname{fin}}(m_2)$ for $m_1,m_2\in \Cf$. Then $S(m_1)=S(m_2)$. By \eqref{eq:mu-u}, the measure $\mu_1$ associated with $m_1$ equals the one associated to $m_2$ on all finite paths, hence $m_1=m_2$. This shows injectivity and finishes the proof.
\end{proof}

Next we analyze the elements in $D^\beta$ in relation to vertices in $\Ecrit$. The first observation is that $\sim$ is an equivalence relation on $\Ecrit$ due to Proposition~\ref{prop:sums}, assertions \eqref{ item:gg3} and \eqref{ item:gg4}.  We let $\Ecrit/_{\sim}$ denote the set of equivalence classes and write $\mathfrak{v}:=\{v'\in \Ecrit\mid v\sim v'\}$ for the equivalence class of $v$.

\begin{theorem}\label{thm:infinitetype-crit} Let $\beta\in [0,\infty)$. The map $\Winf$ from $D^\beta$ to the set of $[0,1]$-valued functions on $\Ecrit/_{\sim}$  given by
\begin{equation}\label{eqdef:W-inf}
	\Winf(m)(\mathfrak{v})=m(v)\Zva(\beta)
\end{equation}
for $\mathfrak{v}\in \Ecrit/_{\sim}$, $v\in \mathfrak{v}$, is a well-defined convex isomorphism from $\Cinfa$ onto $$\biggl\{\epsilon:(\Ecrit/_{\sim})\to [0,1]\biggm\vert \sum_{\mathfrak{v}\in\Ecrit/_{\sim}}\epsilon(\mathfrak{v})=1\biggr\}.$$

For every $m\in D^\beta$, the map $m_{\mathfrak{v}}^\beta:=m_v^\beta$ is well-defined on $\Ecrit/_{\sim}$. The inverse of $\Winf$ is  $\epsilon\mapsto \sum_{\mathfrak{v}\in\Ecrit/_{\sim}}\epsilon(\mathfrak{v})m^\beta_{\mathfrak{v}}$.
\end{theorem}

The proof of this theorem will follow from a series of results. We start by investigating when an equality $m_{v_1}^\beta=m_{v_2}^\beta$ can take place.

\begin{lemma}\label{lem:help}
	Let $v\in \Ecrit\cup \left(\Ereg\setminus R\right)$ and $m\in D^\beta$. Then $m=m^\beta_v$ if and only if $m(v)=(\Zva(\beta))^{-1}$.
\end{lemma}

\begin{proof}
	Suppose that $m(v)=(\Zva(\beta))^{-1}$, and let $\mu$ be the measure associated to $m$  by \eqref{eq:convention-mu-m}. It follows from the scaling condition in $C^\beta$ that $\mu(\cyl{u}\cap\partial_RE)=N(u)^{-\beta}\mu(\cyl{v}\cap\partial_RE)=N(u)^{-\beta}(\Zva(\beta))^{-1}$ for any $u\in E^*v$. Thus
	\begin{equation*}
		\sum_{u\in\Eva}\mu(\cyl{u}\cap\partial_RE)=\sum_{u\in\Eva}N(u)^{-\beta}(\Zva(\beta))^{-1}=1.
	\end{equation*}
	This shows that $A:=\{u\in\Eva\mid \mu(\cyl{u}\cap\partial_RE)>0\}$ is countable. Let $v'\in E^0$. From
	\begin{align*}
		m(v')&=\mu(\cyl{v'}\cap\partial_RE)=\sum_{u\in A}\mu(\cyl{u}\cap\partial_RE)\mu(\cyl{v'}\cap\partial_RE)\\
		&=\sum_{u\in A,\ s(u)=v'}\mu(\cyl{u}\cap\partial_RE)
		=\sum_{u\in A,\ s(u)=v'}N(u)^{-\beta}(\Zva(\beta))^{-1}
	\end{align*}
it follows that $m(v')\leq m^\beta_v(v')$ because $\{u\in A, s(u)=v\}\subset \Es{v'}{v}$.
	Conversely, if $u_1,u_2\in\Es{v'}{v}$ and $u_1\ne u_2$, then $\cyl{u_1}\cap\cyl{u_2}=\emptyset$. Hence we have $\bigcup{u\in\Es{v'}{v}}\cyl{u}\subseteq\cyl{v'}$, so
	\begin{align*}
			m(v')&=\mu(\cyl{v'}\cap\partial_RE)\ge\sum_{u\in\Es{v'}{v}}\mu(\cyl{u}\cap\partial_RE)\\
			&=\sum_{u\in\Es{v'}{v}}N(u)^{-\beta}(\Zva(\beta))^{-1}=m^\beta_v(v').
		\end{align*}
		Thus, $m=m^\beta_v$.
\end{proof}

\begin{lemma}\label{lem:sums}
	Suppose $v_1,v_2\in\Ecrit$ and that $v_1\sim v_2$. Then $$\Bigl(\sum_{u_1\in\Es{v_2}{v_1}}N(u_1)^{-\beta}\Bigr)\Bigl(\sum_{u_2\in\Es{v_1}{v_2}}N(u_2)^{-\beta}\Bigr)=1.$$
\end{lemma}

\begin{proof}
	Let $P=\{u\in v_2E_s^*v_2\mid r(u')\ne v_1\text{ for any }u'\le u\}$ and $x=\sum_{u\in P}N(u)^{-\beta}$. Then the
assumption $v_1\sim v_2$ implies that
$$
x<\sum_{u\in v_2E_s^*v_2}N(u)^{-\beta}=Z^s_{v_2}(\beta)=1.
 $$
 Now $(u_1,\dots,u_n,u)\mapsto u_1\dotsm u_nu$ defines a bijection between $\bigcup_{n=0}^\infty(P)^n\times (v_2E_s^*v_2\setminus P)$ (where $(P)^0=\{v_2\}$) and $\{u_1u_2\mid u_1\in\Es{v_2}{v_1},\ u_2\in\Es{v_1}{v_2}\}$, hence
	\begin{equation*}
		\Bigl(\sum_{u_1\in\Es{v_2}{v_1}}N(u_1)^{-\beta}\Bigr)\Bigl(\sum_{u_2\in\Es{v_1}{v_2}}N(u_2)^{-\beta}\Bigr)=\sum_{n=0}^\infty x^n(1-x)=1.
	\end{equation*}
\end{proof}

\begin{proposition}\label{prop:meq}
	Suppose $v_1,v_2\in \Ecrit\cup \left(\Ereg\setminus R\right)$ and $v_1\ne v_2$. Then $m^\beta_{v_1}=m^\beta_{v_2}$ if and only if $v_1,v_2\in\Ecrit$ and $v_1\sim v_2$.
\end{proposition}

\begin{proof}
	Assume $v_1,v_2\in\Ecrit$ and $v_1\sim v_2$. Two applications of Proposition \ref{prop:E}\eqref{item:a3.5} give us that
	\begin{align*}
		&\Bigl(\sum_{u_1\in\Es{v_2}{v_1}} N(u_1)^{-\beta}\Bigr)\Bigl(\sum_{u_2\in\Es{v_1}{v_2}} N(u_2)^{-\beta}m^\beta_{v_1}(v_2)\Bigr)\\
		&\qquad\le\Bigl(\sum_{u_1\in\Es{v_2}{v_1}} N(u_1)^{-\beta}\Bigr)\Bigl(\sum_{u_2\in\Es{v_1}{v_2}}N(u_2)^{-\beta}(Z^a_{v_2}(\beta))^{-1}\Bigr)\\
		&\qquad=\sum_{u_1\in\Es{v_2}{v_1}} N(u_1)^{-\beta}m^\beta_{v_2}(v_1)
		\le \sum_{u_1\in\Es{v_2}{v_1}} N(u_1)^{-\beta}(Z^a_{v_1}(\beta))^{-1}\\
		&\qquad=m^\beta_{v_1}(v_2).
	\end{align*}
	It follows from Lemma \ref{lem:sums} that the above inequalities are in fact equalities, so $m^\beta_{v_1}(v_2)=(Z^a_{v_2}(\beta))^{-1}$. Thus $m^\beta_{v_1}=m^\beta_{v_2}$ by Lemma \ref{lem:help}.
	
	If $v_1\in\Ereg$, then Proposition \ref{prop:m} implies that $S(m^\beta_{v_1})(v_1)\ne 0$. Since $S(m^\beta_{v_2})(v_1)=0$, necessarily then $m^\beta_{v_1}\ne m^\beta_{v_2}$. Similarly, $m^\beta_{v_1}\ne m^\beta_{v_2}$ if $v_2\in\Ereg$.
	
	If $v_1\not\succ v_2$, then $m^\beta_{v_2}(v_1)=0\ne (Z^a_{v_1}(\beta))^{-1}=m^\beta_{v_1}(v_1)$, so $m^\beta_{v_1}\ne m^\beta_{v_2}$. Similarly, $m^\beta_{v_1}\ne m^\beta_{v_2}$ if $v_2\not\succ v_1$.
\end{proof}

\begin{proposition}\label{prop:equ}
	Let $m\in D^\beta$ and $v_1,v_2\in\Ecrit$. Suppose $v_1\sim v_2$. Then $m(v_1)Z^a_{v_1}(\beta)=m(v_2)Z^a_{v_2}(\beta)$.
\end{proposition}

\begin{proof}
	Proposition \ref{prop:meq} and Lemma \ref{lem:help} imply that
	\begin{equation} \label{eq:t}
	(Z^a_{v_1}(\beta))^{-1}=\sum_{u\in\Es{v_1}{v_2}}N(u)^{-\beta}(Z^a_{v_2}(\beta))^{-1}.
	\end{equation}
	
	 It follows from Lemma \ref{lem:mineq} that $m(v_2)\ge \sum_{u\in\Es{v_2}{v_1}}N(u)^{-\beta}m(v_1)$ and, similarly, that $m(v_1)\ge \sum_{u\in\Es{v_1}{v_2}}N(u)^{-\beta}m(v_2)$, so
	\begin{multline*}
		\Bigl(\sum_{u_1\in\Es{v_2}{v_1}} N(u_1)^{-\beta}\Bigr)\Bigl(\sum_{u_2\in\Es{v_1}{v_2}} N(u_2)^{-\beta}m(v_2)\Bigr)\\\le\sum_{u_1\in\Es{v_2}{v_1}} N(u_1)^{-\beta}m(v_1)\le m(v_2).
	\end{multline*}
	According to Lemma \ref{lem:sums}, the above inequalities are in fact equalities, so
	 $$m(v_1)=\sum_{u\in\Es{v_1}{v_2}}N(u)^{-\beta}m(v_2)=\frac{Z^a_{v_2}(\beta)}{Z^a_{v_1}(\beta)}m(v_2)$$
	where the last equality follows from Equation \eqref{eq:t}.
\end{proof}

\begin{proof}[Proof of Theorem~\ref{thm:infinitetype-crit}.]
We start by showing that the maps introduced in the formulation of the theorem are well-defined.
Let $m\in D^\beta$ and $\mathfrak{v}\in \Ecrit/_{\sim}$. We deduce from Proposition \ref{prop:meq} that  $m_{v_1}^\beta=m_{v_2}^\beta$ for $v_1,v_2\in \mathfrak{v}$. Therefore, $m_{\mathfrak{v}}^\beta:=m_v^\beta$ for $v\in\mathfrak{v}$ is well-defined.
It follows from Proposition \ref{prop:equ} that the quantity $m(v)\Zva(\beta)$ does not depend on the choice of $v\in \mathfrak{v}$.

Suppose $m\in \Cinfa$ and $v\in\Eequ$. Let $R_v$ denote the set of infinite paths that start in $v$ and return to $v$ infinitely often. Then  $\mu(R_v)=\lim_{n\to\infty}(\Zvav(\beta))^n m(v)$, where $\mu$ is the measure associated with $m$. Hence $\mu(R_v)=0$ for $v\in\Ereg$ and
$\mu(R_v)=m(v)$ when $v\in\Ecrit$. Now let $T_v$ be the collection of infinite paths that pass through $v$ infinitely often. It follows that  $\mu(T_v)=\Zva(\beta) \mu(R_v)$, so $\mu(T_v)$ is zero for $v\in\Ereg$, and equals
$\Zva(\beta)m(v)$ when $v\in\Ecrit$. Suppose now that $v_1\sim v_2$. Let $P$ be as in the proof of Lemma~\ref{lem:sums}. Then
 $\sum_{u\in P}N(u)^{-\beta}<1$, and therefore $\mu(T_{v_1}\bigtriangleup T_{v_2})=0$ where
$T_{v_1}\bigtriangleup T_{v_2}=(T_{v_1}\setminus
T_{v_2})\cup(T_{v_2}\setminus T_{v_1})$. Since $\mu$ is conservative, it has support on $\Erec$. We conclude that
$$
\sum_{\mathfrak{v}\in\Ecrit/_{\sim}}
m(v)\Zva(\beta)=\sum_{\mathfrak{v}\in\Ecrit/_{\sim}} \mu(T_v)=1,
$$
where $v$ is taken such that $v\in \mathfrak{v}$ as $\mathfrak{v}$ runs over $\Ecrit/_{\sim}$.

	Clearly $W_{\operatorname{inf}}$ is a convex map. We  show next that
	\begin{equation} \label{eq:supset}
		\bigl\{\epsilon:\Ecrit/_{\sim}\to [0,1]\bigm\vert \sum_{\mathfrak{v}\in\Ecrit/_{\sim}}\epsilon(\mathfrak{v})=1\bigr\}\subseteq \Winf(\Cinfa).
	\end{equation}
 For $\mathfrak{v}\in \Ecrit/_{\sim}$ let $\delta_{\mathfrak{v}}:  \Ecrit/_{\sim}\to [0,1]$ be defined by $\delta_{\mathfrak{v}}(\mathfrak{v})=1$ and $\delta_{\mathfrak{v}}(\mathfrak{v}')=0$ for $\mathfrak{v}'\in \Ecrit/_{\sim}$ different from $\mathfrak{v}$. Then Lemma \ref{lem:help}  and 	Proposition \ref{prop:sums}\eqref{ item:gg4} show that
 $\Winf(m^\beta_{\mathfrak{v}})=\delta_{\mathfrak{v}}$ for $\mathfrak{v}\in\Ecrit/_{\sim}$, which gives the claimed set inclusion.

To prove that $\Winf$ is injective we show that
	\begin{equation} \label{eq:inj}
		\sum_{\mathfrak{v}\in\Ecrit/_{\sim}}\Winf(m)(\mathfrak{v})m^\beta_{\mathfrak{v}}=m
	\end{equation}
	for any $m\in \Cinfa$. Denote by $m'$ the term $\sum_{\mathfrak{v}\in\Ecrit/_{\sim}}\Winf(m)(\mathfrak{v})m^\beta_{\mathfrak{v}}$. Proposition~\ref{prop:sums}, Lemma~\ref{lem:help}   and Proposition~\ref{prop:equ} imply that
 $m(v)=m'(v)$ for $v\in\Ecrit\setminus R$. On the other hand, Lemma~\ref{lem:mineq} shows that
  $m(v)\ge m'(v)$ for all $v\in E^0$.  Since $\sum_{v\in
E^0}m(v)=\sum_{v\in E^0}m(v')=1$ we must have $m(v)=m'(v)$ for all $v\in E^0$.

\end{proof}

When $A$ is a subset of $D^\beta$ we let $\conv A$ denote the convex hull of $A$. We then have the following consequence of Theorems~\ref{thm:finitetype-reg} and \ref{thm:infinitetype-crit}.

\begin{corollary}\label{cor:extreme}
	\begin{enumerate}
		\item $\{m^\beta_v\mid v\in\Ereg\setminus R\}$ is the set of extreme points of $\Cf$, and $\Cf=\conv \{m^\beta_v\mid v\in\Ereg\setminus R\}$.
		\item $\{m^\beta_{\mathfrak{v}}\mid \mathfrak{v}\in\Ecrit/_{\sim}\}$ is the set of extreme points of $\Cinfa$, and $\Cinfa=\conv\{m^\beta_{\mathfrak{v}}\mid \mathfrak{v}\in\Ecrit/_{\sim}\}$.
		\item If $\Cinfb=\emptyset$ (in particular if $\Ewan=\emptyset$), then $\{m^\beta_v\mid v\in\Ereg\setminus R\}\cup\{m^\beta_{\mathfrak{v}}\mid \mathfrak{v}\in\Ecrit/_{\sim}\}$ is the set of extreme points of $D^\beta$, and $D^\beta=\conv \{m^\beta_v\mid v\in\Ereg\setminus R\}\cup\{m^\beta_{\mathfrak{v}}\mid \mathfrak{v}\in\Ecrit/_{\sim}\}$.
	\end{enumerate}
\end{corollary}

Notice that if $R\ne \emptyset$, then the KMS states of $(\mathcal{T}C^*(E),\sigma)$ that descend to KMS states on $(C^*(E,R),\sigma)$ are the ones that correspond to elements of $\conv\{m^\beta_v\mid v\in\Ereg\setminus R\}\cup\Cinf$.

\begin{remark} Proposition \ref{prop:sums} implies that for arbitrary $v_1,v_2$ in $\Eequ$ with $v_1\sim v_2$,  either $v_1,v_2$
both belong to $\Ecrit$ or they both belong to $\Ereg$. In particular, if the graph is connected in the sense that for every $v_1$ and $v_2$ in $E^0$ we have $v_1\sim v_2$, then for each $\beta\geq 0$,
either $\Eequ=\Ecrit$ or $\Eequ=\Ereg$. Thus at each $\beta\geq 0$, either $\Cinfa$ or  $\Cf$ are empty. If the graph is not connected, it may happen that both $\Cinfa$ and $\Cf$ are non-trivial at some $\beta\geq 0$, see Example~\ref{eq:stable-TOn}. It might also happen that both $\Cinfb$ and $\Cf$ are non-trivial at some $\beta\geq 0$, see Theorem \ref{thm:exmotivating}.
\end{remark}

We conclude this section with a couple of general remarks about existence of KMS states. Recall that for $\beta\in [0,\infty)$ and $v\in E^0$, we defined the partition function with fixed-target $Z_v(\beta)$ in  \eqref{eq:Zv}. If there is  $\beta$ in $[0,\infty)$ such that $Z_v(\beta)<\infty$, we set
$$\beta_v=\inf\{\beta\in [0,\infty)\mid Z_v(\beta)<\infty\}.
$$
Otherwise we let $\beta_v=\infty$. We have the following simple observation.
\begin{proposition}\label{thm:no-kms}
	Let $v\in E^0$. If $\beta<\beta_v$, then there is no KMS$_\beta$ state $\psi$ for $(C^*(E,R),\sigma)$ such that $\psi(p_v)\ne 0$.
\end{proposition}

\begin{proof}
	Suppose $\psi$ is a KMS$_\beta$ state for $(C^*(E,R),\sigma)$ such that $\psi(p_v)\ne 0$. It follows from Proposition~\ref{prop:E}\eqref{item:a4} that $v\in\Eequ$ and thus that $v\in E^0_{\beta'\text{-reg}}$ for any $\beta'>\beta$. It then follows from Proposition~\ref{prop:E}\eqref{item:a2} that $\beta\ge \beta_v$.
\end{proof}

The following observation can be helpful in computing KMS$_\beta$ states for particularly nice graphs and will be used in Examples \ref{ex:a}, \ref{ex:b}, and \ref{ex:c}.

\begin{proposition}\label{prop:nice-graphs}
	Suppose that $N(e)=\exp(1)$ for all $e\in E^1$ and that there exist $k,l\in\N$ such that every $v\in E^0$ receives exactly $k$ paths of length $l\geq 1$. Then:
	\begin{enumerate}
		\item $D^\beta=\emptyset$ when $\beta<\ln (k)/l$.
		\item $D^\beta=D^{\beta}_{\operatorname{inf}}$ when $\beta=\ln (k)/l$.
		\item $\{m^\beta_v\mid v\in E^0\}$ are the extreme points of $D^\beta$ and $D^\beta=\Cf$ when $\beta>\ln (k)/l$.
	\end{enumerate}
\end{proposition}

\begin{proof}
	Suppose $m\in D^\beta$ is of infinite type. Thus $m(v)=\sum_{e\in vE^1}N(e)^{-\beta}m(r(e))$ for all $v\in E^0$, according to Lemma~\ref{lem:aboutS}(b).  Iterations of this equality imply that
$$
m(v)=\sum_{u\in vE^l}N(u)^{-\beta}m(r(u))=\exp(-l\beta)\sum_{u\in vE^l}m(r(u))
$$
for every $v\in E^0$. Hence	
\begin{align*}
		1&=\sum_{v\in E^0}m(v)=\exp(-l\beta)\sum_{v\in E^0}\sum_{u\in vE^l}m(r(u))\\
 &=\exp(-l\beta)\sum_{v'\in E^0}km(v')=k\exp(-l\beta)
	\end{align*}
	from which it follows that $\beta=\ln(k)/l$. Thus $D^{\beta}_{\operatorname{inf}}=\emptyset$ for $\beta\ne \ln(k)/l$.
	
	For each $v\in E^0$, the partition function $Z_v(\beta)$ at $v$ satisfies that $(Z_v(\beta))^l=\sum_{n=0}^\infty k^n\exp(-nl\beta)$. It thus follows from Proposition \ref{prop:E} that $\Ereg=\emptyset$ for $\beta\le \ln (k)/l$, and that $\Ereg=E^0$ for $\beta>\ln (k)/l$. Hence $D^\beta=\emptyset$ for $\beta<\ln (k)/l$, $D^\beta_4=D^{\beta}_{\operatorname{inf}}$ when $\beta=\ln (k)/l$, and the extreme points of $D^\beta=\Cf$ are $\{m^\beta_v\mid v\in E^0\}$ when $\beta>\ln (k)/l$.
\end{proof}

\section{Ground states and KMS$_\infty$ states}\label{section:ground}

Note that the definition of $D^{\operatorname{gr}}$ implies directly that the set of its extreme points  is $\{m^{\operatorname{gr}}_v\mid v\in E^0\setminus R\}$, where $m^{\operatorname{gr}}_v:E^0\to [0,1]$ is defined by
\begin{equation}\label{eq:ground-pointmass}
m^{\operatorname{gr}}_v(v')=\begin{cases}1&\text{if }v'=v,\\0&\text{if }v'\ne v.\end{cases}
\end{equation}
Thus, we have a complete concrete description of all the ground states of $(C^*(E,R),\sigma)$.

A ground state is called a \emph{KMS$_\infty$ state} if it is the weak* limit of a sequence of KMS$_{\beta_n}$ states as $\beta_n\to\infty$ (see \cite{Con-Mar} and \cite[\S 1]{aHLRS}). We will now characterize which of the ground states of $(C^*(E,R),\sigma)$ are KMS$_\infty$ states. Since for a finite graph $E$ we have $\beta_v<\infty$ for all $v\in E^0$, the  next result  generalizes \cite[Proposition 5.1]{aHLRS}.

\begin{theorem}\label{thm:infty-states}
    Given a directed graph $E$, a subset $R$ of $E^0_\reg$ and a function $N:E^1\to (1,\infty)$, let $\sigma$ be the strongly continuous one-parameter group
     of automorphisms of $C^*(E,R)$ such that
    \begin{equation*}
  	  \sigma_t(s_e)=\bigl(N(e)\bigr)^{it}s_e \text{ and }\sigma_t(p_v)=p_v
  	\end{equation*}
    for all $e\in E^1$ and $v\in E^0$.
	
	A ground state $\psi$ of $(C^*(E,R),\sigma)$ is a KMS$_\infty$ state if and only if $\beta_v<\infty$ for every $v\in E^0$ for which $\psi(p_v)\ne 0$.
\end{theorem}

\begin{proof}
	For $m\in D^\beta$ let $\psi_m$ be the KMS$_\beta$ state corresponding to $m$.
	
	Assume first that $\psi$ is a KMS$_\infty$ state and that $\psi(p_v)\ne 0$. Then there is a $\beta<\infty$ and a KMS$_\beta$ state which is non-zero on $p_v$. It follows from Proposition \ref{thm:no-kms} that $\beta_v<\infty$. This shows the only if direction.
	
	For the converse direction, since $\{\psi_{m^{\operatorname{gr}}_v}\mid v\in E^0\setminus R\}$ are the extreme points of $A^{\operatorname{gr}}$, it suffices to show that $\psi_{m^{\operatorname{gr}}_v}$ is a KMS$_\infty$ state if $v\in E^0\setminus R$ and $\beta_v<\infty$. We will establish this by showing that $(\psi_{m^\beta_v})$ converges to $\psi_{m^{\operatorname{gr}}_v}$ in the weak*-topology as $\beta\to\infty$.
	
	We have that $\sum_{u\in\Eva}N(u)^{-\beta}=\Zva(\beta)<\infty$ for $\beta>\beta_v$, and since $N(u)^{-\beta}$ converges monotonically to 0 as $\beta\to\infty$ for $u\in\Eva\setminus\{v\}$, an application of the monotone convergence theorem yields that $\Zva(\beta)\to 1$ as $\beta\to\infty$. A similar argument gives us that $\sum_{u\in v'E_a^*v}N(u)^{-\beta}\to 0$ as $\beta\to\infty$ for $v'\in E^0\setminus\{v\}$. Thus $m^\beta_v(v')$ converges pointwise to $m^{\operatorname{gr}}_v(v')$ as $\beta\to\infty$, for each $v'\in E^0$. This implies our claim that $(\psi_{m^\beta_v})$ converges to $\psi_{m^{\operatorname{gr}}_v}$ in the weak*-topology as $\beta\to\infty$.
\end{proof}

Example \ref{ex:d} provides an example of a ground state which is not a KMS$_\infty$ state.

\section{Examples}

All throughout this section we let $\No=\{0,1,2,\dots\}$.

\begin{example} Our first example is a graph where Theorems~\ref{thm:finitetype-reg} and \ref{thm:infinitetype-crit} describe completely the KMS states of $\mathcal{T}C^*(E)$ endowed with the gauge action. Let $E$ be the graph with $E^0=\{v_n\mid n\in \No\}$ and $E^1=\{e_n, f_n\mid n\in \No\}$ given by $s(e_n)=v_0=r(e_0)=r(f_0)$ for all $n\geq 0$, $r(e_n)=v_n=r(f_{n})$ for $n\geq 1$, and $s(f_n)=v_{n-1}$ for $n\geq 1$:
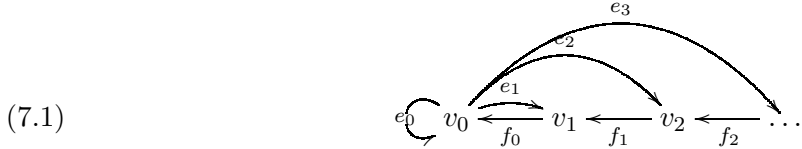
\begin{equation}
  \xymatrix{
{v_0} \ar@(ul,dl)[]|{e_0} \ar@/^/[r] ^-{e_1} \ar@/^ 2pc/[rr] ^-{e_2} \ar@/^ 3pc/[rrr]^{e_3} &{v_1} \ar[l]^{f_0} &{v_2} \ar[l]^{f_1}&{\dots}\ar[l]^{f_2}
  }
\end{equation}
Let $R=\emptyset$ and $N:E^1\to (1,\infty)$ be $N(e_n)=N(f_n)=\exp(1)$ for all $n\in \No$. Thus we are dealing with $\mathcal{T}C^*(E)$ and its gauge action. Since every infinite path passes through $v_0$ infinitely many times, we have $\Erec=E^\infty$ and $\Ewan=\emptyset$. Hence $\Cinf=\Cinfa$. The partition functions at $v_0$ are  given as follows:
\begin{align*}
Z^s_{v_0}(\beta)
&= \sum_{k=1}^\infty \exp(-k\beta)=\frac {\exp(-\beta)}{1-\exp(-\beta)}\text{ and }\\
Z^a_{v_0}(\beta)
&= \sum_{k=0}^\infty \exp(-k\beta)=\frac {1}{1-\exp(-\beta)}.
\end{align*}

Thus for $\beta\in [0,\ln 2)$ we have $\Eequ=\emptyset$, for $\beta=\ln 2$ we have $\Ecrit=E^0$, and for $\beta\in  (\ln 2,\infty)$, $\Ereg=E^0$. Hence $D^\beta=\emptyset$ for $\beta\in [0,\ln 2)$.
Further, by Theorem~\ref{thm:infinitetype-crit},
$$
D^\beta=\Cinfa=\{m_{E^0}^\beta\} \text{ when }\beta=\ln 2.
$$
Finally, for  $\beta\in  (\ln 2,\infty)$ Theorem~\ref{thm:finitetype-reg} gives that $D^\beta=\Cf$ is isomorphic as a convex set with the set of functions $\{\epsilon:E^0\to [0,1]\mid \sum_{v\in E^0}=1\}$. The extremal points of $D^\beta$ coincide with $\{m^\beta_{v_n}\mid n\in\No\}$, and every $m\in D^\beta$ has form
 $m=\sum_{n=0}^\infty\frac{S(m)(v_n)Z^a_{v_n}(\beta)}{1-Z^s_{v_n}(\beta)}m^\beta_{v_n}$.

It is easy to check that $m^{\ln 2}_{E^0}(v_n)=2^{-n-1}$ for all $n\in\No$. It is not difficult, but a bit tedious to write down explicit formulas for $m^\beta_v$ when $\beta>\ln 2$ and $v\in E^0$.

We have that $D^{\operatorname{gr}}=\conv\{m^{\operatorname{gr}}_v\mid v\in E^0\}$, and it follows from Theorem \ref{thm:infty-states} that every ground state is a KMS$_\infty$ state.

Since $E^0_\reg=E^0\setminus\{v_0\}$, it follows that the only KMS states of $\mathcal{T}C^*(E)$ that descend to KMS states on $C^*(E)$ are the ones corresponding to $m^{\ln 2}_{E^0}$ and $m^\beta_{v_0}$, $\beta>\ln 2$. The only ground state that descends to $C^*(E)$ is the one corresponding to $m^{\operatorname{gr}}_{v_0}$.
\end{example}

\begin{example} \label{ex:a}
	Next we introduce an example of a strongly connected graph $E$ with finite degree (or valence) for which $(C^*(E),\sigma)$ has no KMS states when $\sigma$ is the gauge action.
	
The graph $E$ is defined as follows
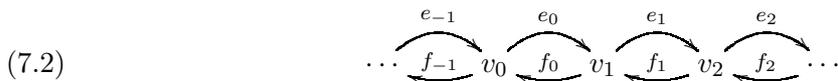
\begin{equation}
  \xymatrix{
{\cdots} \ar@/^ 1pc/[r] ^-{e_{-1}}   &{v_0} \ar@/^/[l] _-{f_{-1}}  \ar@/^ 1pc/[r] ^-{e_{0}}   &{v_1} \ar@/^ 1pc/[r] ^-{e_{1}}  \ar@/^/[l] _-{f_{0}}   &{v_2} \ar@/^ 1pc/[r] ^-{e_{2}}  \ar@/^/[l] _-{f_1}  & {\cdots} \ar@/^/[l] _-{f_2}
  }
\end{equation}
Let $R=\emptyset$ and $N:E^1\to (1,\infty)$ be $N(e)=\exp(1)$ for all $e\in E^1$.

Proposition~\ref{prop:nice-graphs} implies that $D^\beta=\emptyset$ for $\beta<\ln 2$, $D^\beta=D^\beta_{\operatorname{inf}}$ when $\beta=\ln (2)$, and that $D^\beta=\Cf=\conv\{m^\beta_v\mid v\in E^0\}$ for $\beta>\ln 2$.

Suppose $m\in D^{\ln 2}_{\operatorname{inf}}$. Let $n\in \mathbb{Z}$ and denote $a=m(v_n)$. Since $S(m)(v)=0$ for all $v$, it follows that  $a=1/2(m(v_{n-1})+m(v_{n+1}))$. Thus, either $m(v_{n-1})\geq a$ or $m(v_{n+1})\geq a$. By symmetry of the graph, we may assume $m(v_{n+1})\geq a$. Let $b:=m(v_{n+1})$. By induction on $k\geq 1$, $m(v_{n+k})=kb-(k-1)a\geq a$. However, $\sum_{k=0}^\infty m(v_{n+k})\leq \sum_{v\in E^0} m(v)=1$, so necessarily $a=0$. Since $n$ was arbitrarily chosen, this shows that $m\equiv 0$, a contradiction. Thus $D^{\ln 2}_{\operatorname{inf}}=\emptyset$.

We conclude that $D^\beta=\emptyset$ for $\beta\le\ln (2)$, and that  $D^\beta=\Cf=\conv\{m^\beta_v\mid v\in E^0\}$ for $\beta>\ln 2$. We furthermore have that $D^{\operatorname{gr}}=\conv\{m^{\operatorname{gr}}_v\mid v\in E^0\}$, and it follows from Theorem \ref{thm:infty-states} that every ground state is a KMS$_\infty$ state.

Since $E^0_\reg=E^0$, none of the KMS or ground states of $\mathcal{T}C^*(E)$ descend to KMS or ground states of $C^*(E)$, so the analogue of \cite[Theorem 4.3]{aHLRS}  does not hold for infinite graphs even under the assumption that $E$ has finite degree (or valence).
\end{example}

\begin{example}\label{ex:dis}
We now present an example where the set of dissipative measures in non-empty.

Let $E$ be the graph with $E^0=\{v_n\mid n\in \No\}$ and $E^1=\{e_n\mid n\in \No\}\cup\{f_n\mid \No\}$ where $s(e_n)=s(e_n)=v_n$ and $r(e_n)=r(f_n)=v_{n+1}$, see the  picture:
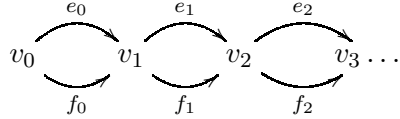
\begin{equation*}
  \xymatrix{
  {v_0} \ar@/^ 1pc/[r]^{e_0} \ar@/_ 1pc/[r]_{f_0} &{v_1}\ar@/^ 1pc/[r]^{e_1} \ar@/_ 1pc/[r]_{f_1} &{v_2}\ar@/^ 1pc/[r]^{e_2} \ar@/_ 1pc/[r]_{f_2}&{v_3\dots}
  }
\end{equation*}
Let $R=\emptyset$ and $N:E^1 \to (1,\infty)$ be $N(e)=\exp(1)$ for all $e\in E^1$. Thus we are dealing with $\mathcal{T}C^*(E)$ and its gauge action.	

It is easy to see that $\Evav=\emptyset$ and that $\Eva$ is finite for every $v\in E^0$. It follows that $\Zvav(\beta)=0$ and $\Zva(\beta)<\infty$ for all $v\in E^0$ and all $\beta\in [0,\infty)$. Thus, $\Ereg=E^0$ for all $\beta\in [0,\infty)$. It follows that $\Cf=\{m_v^\beta\mid v\in E^0\}$ and $\Cinfa=\emptyset$ for all $\beta\in [0,\infty)$.

Suppose $\beta\in [0,\ln 2)$.  Define $m^\beta:E^0\to [0,1]$ by $m^\beta(v_n)=(1-\exp(\beta)/2)(\exp(\beta)/2)^n$ for $n\in\No$. Then
\begin{align*}
	m^\beta(v_n)&=(1-\exp(\beta)/2)(\exp(\beta)/2)^n=2\exp(-\beta)(1-\exp(\beta)/2)(\exp(\beta)/2)^{n+1}\\
	&=\sum_{e\in v_nE^1}(N(e))^{-\beta}m^\beta(r(e))
\end{align*}
for all $n\in\No$, and
\begin{equation*}
	\sum_{v\in E^0}m^\beta(v)=\sum_{n=0}^\infty (1-\exp(\beta)/2)(\exp(\beta)/2)^n=1.
\end{equation*}
Thus $m^\beta\in\Cinf=\Cinfb$.

Let $\beta\in [0,\infty)$ and suppose $m\in \Cinf=\Cinfb$. Then
\begin{equation*}
	m(v_n)=\sum_{e\in v_nE^1}(N(e))^{-\beta}m(r(e))=2\exp(-\beta)m(v_{n+1})
\end{equation*}
for all $n\in\No$. It follows that $m(v_n)=(\exp(\beta)/2)^nm(v_0)$ for all $n\in\No$. Since $$1=\sum_{v\in E^0}m(v)=\sum_{n=0}^\infty (\exp(\beta)/2)^nm(v_0),$$
it follows that $\beta\in [0,\ln 2)$ and that $m(v_0)=1-\exp(\beta)/2$, and thus that $m=m^\beta$.

Thus $\Cinf=\Cinfb=\{m^\beta\}$ for $\beta\in [0,\ln 2)$, and $\Cinf=\emptyset$ for $\beta\ge\ln 2$. Since $E^0_\reg=E^0$, the only KMS$_\beta$ states that descend to $C^*(E)$ are the ones corresponding to $m^\beta$ for $\beta\in [0,\ln 2)$. It may be of interest to observe that  for each $k\in \No$, the sequence $\{m^\beta_{v_n}(v_k)\}_{n}$ converges to $m^\beta(v_k)$.
\end{example}

\begin{example}\label{ex:motivating}

In Example \ref{ex:dis} we presented an example where $\Cinf=\Cinfb\ne\emptyset$ for $\beta\in [0,\ln 2)$, and $\Cinf=\emptyset$ for $\beta\ge\ln 2$. We now present an example where $\Cinf=\Cinfa\ne\emptyset$ for $\beta=0$, $\Cinf=\Cinfb\ne\emptyset$ for $\beta\in (0,\ln 2)$, and $\Cinf=\emptyset$ for $\beta\ge\ln 2$.

Let $E$ be the graph with $E^0=\{v_n\mid n\in \No\}$ and $E^1=\{e_n\mid n\in \No\}\cup\{f_n\mid n\in \No\}$ where $r(e_n)=s(e_n)=v_n$ and $s(f_n)=v_n$ and $r(f_n)=v_{n+1}$ for $n\in \No$, see the  picture:
\begin{equation}\label{graph-motivating}
  \xymatrix{
  {v_0} \ar@(ul,ur)[]|{e_0} \ar[r]^{f_0} &{v_1}\ar@(ul,ur)[]|{e_1} \ar[r]^{f_1} &{v_2}\ar@(ul,ur)[]|{e_2} \ar[r]^{f_2}&{\dots}
  }
\end{equation}
Let $R=\emptyset$ and $N:E^1 \to (1,\infty)$ be $N(e)=\exp(1)$ for all $e\in E^1$. Thus we are dealing with $\mathcal{T}C^*(E)$ and its gauge action.

Fix  $\beta\in [0,\infty)$. We have $v_nE_s^*v_n=\{e_n\}$ for all $n\in\No$. It follows that $$Z^s_{v_n}(-\beta)=\exp(-\beta);
$$
note in particular that this is independent of  $v\in E^0$.

Assume now that $n>0$. Then $$\Es{v_{n-1}}{v_n}=\{e_{n-1}^kf_{n-1}\mid k\in\No\}$$ where $e_{n-1}^k$ is the path we get by concatenating $e_{n-1}$ with itself $k$ times. It follows that
$$
\sum_{u\in \Es{v_{n-1}}{v_n}}N(u)^{-\beta}=\sum_{k=0}^\infty(\exp(-\beta))^k
$$ diverges to infinity if $\beta=0$, and is convergent with sum $\exp(-\beta)/(1-\exp(-\beta))$ if $\beta>0$. Assume $\beta>0$ and let $a=\exp(-\beta)/(1-\exp(-\beta))$. If $k<n$, then $(u_1,u_2,\dots,u_{n-k})\mapsto u_1u_2\dots u_{n-k}$ is a bijection $$
\Es{v_k}{v_{k+1}}\times \Es{v_{k+1}}{v_{k+2}}\times\dots \Es{v_{n-1}}{v_{n}}\times\to
\Es{v_k}{v_{n}}.
$$
Hence $$\sum_{u\in \Es{v_k}{v_{n}}}N(u)^{-\beta}=a^{n-k},$$
and
\begin{equation}\label{eq-ex-Zaper}
	Z^a_{v_n}(\beta)=\sum_{k=0}^n\sum_{u\in \Es{v_k}{v_{n}}}N(u)^{-\beta}=\sum_{k=0}^na^{n-k}<\infty.
\end{equation}
In conclusion, we have
\begin{align*}
\Ecrit&={\begin{cases}\{v_0\} &\text{ if } \beta=0\\ \emptyset &\text{ if }\beta>0,\end{cases}}\,\,\text{ and }\,\,
\Ereg={\begin{cases} \emptyset &\text{ if }\beta=0\\ E^0 &\text{ if }\beta>0. \end{cases}}
\end{align*}
 For $\beta=0$, it follows from Proposition \ref{prop:m} that $m^0_{v_0}\in D^\beta$. According to Proposition \ref{prop:E}(4), every $m\in D^0$ must satisfy that $m(v_n)=0$ for $n>0$. Thus $D^0=\{m^0_{v_0}\}$.

 Next we look at positive values of $\beta$. Fix therefore $\beta>0$. Since $\Ereg=E^0$, it follows from Theorem~\ref{thm:finitetype-reg} that $\Cf=\conv\{m^\beta_v\mid v\in E^0\}$. Suppose now that $m\in \Cinf$. Since $\Ecrit=\emptyset$, Theorem~\ref{thm:infinitetype-crit} shows that $\Cinfa=\emptyset$ for every
$\beta\in (0,\infty)$. Thus what is left in order to complete our analysis is to investigate existence of elements in $\Cinfb$. If
$m\in \Cinf$,  Lemma~\ref{lem:aboutS}(b) implies that
$$
m(v_{n+1})=\frac{1-\exp(-\beta)}{\exp(-\beta)}m(v_n)
 $$ for all $n\in\No$. It follows that $\beta<\ln 2$ because otherwise $\frac{1-\exp(-\beta)}{\exp(-\beta)}\ge 1$, which would imply that $\sum_{n=0}^\infty m(v_n)=\infty$.

 We will show that for each $\beta\in (0,\ln 2)$ there is an element in $\Cinfb$. Given $\beta\in (0,\ln 2)$, we have  $a=\exp(-\beta)/(1-\exp(-\beta))>1$, hence $\sum_{n=0}^\infty a^{-n}=\frac{a}{a-1}$.  Notice that if $k<n$, then
\begin{align}
	m^\beta_{v_n}(v_k)&=\sum_{u\in v_{k}E_a^*v_n}N(u)^{-\beta}(Z^a_{v_n}(\beta))^{-1}\notag\\
	&=\frac{a^{n-k}}{\sum_{i=0}^na^{n-i}}\label{eq:m-vn-at-vk}\\
&=\frac{a^{-k}(a-1)}{a-a^{-n}}.\notag
\end{align}
On the other hand, $v_kE_a^*v_n=\emptyset$ for all $k>n$, and thus $m^\beta_{v_n}(v_k)=0$ if $k>n$.
By the proof of Proposition~\ref{prop:m}, $m^\beta_{v_n}(v_n)=(Z^a_{v_n}(\beta))^{-1}$. Hence we
see from \eqref{eq-ex-Zaper} that \eqref{eq:m-vn-at-vk} is valid for all $k=0, 1, \dots, n$, and we in fact have $
\sum_{k=0}^\infty m^\beta_{v_n}(v_k)=\sum_{k=0}^n m^\beta_{v_n}(v_k)=1$ for all  $n\geq 0$.

We now define $m^\beta_{\operatorname{inf}}:E^0\to [0,1]$ by
$$
m^\beta_{\operatorname{inf}}(v_k)= a^{-(k+1)}(a-1)
$$
for all $k\geq 0$. Since  $\sum_{v\in E^0}m^\beta_{\operatorname{inf}}(v)=\sum_{k=0}^\infty a^{-(k+1)}(a-1)=1$, the function
$m^\beta_{\operatorname{inf}}$ satisfies \eqref{item:m1}. Condition \eqref{item:m2} is vacuous, and \eqref{item:m3} is an equality
at all $v\in E^0$, as may be easily verified. Thus $m^\beta_{\operatorname{inf}}\in D^\beta$ and $S(m^\beta_{\operatorname{inf}})=0$. Hence by Lemma~\ref{lem:aboutS}, $m^\beta_{\operatorname{inf}}\in\Cinf$. That $m^\beta_{\operatorname{inf}}\in\Cinfb$ is seen because
the support of the measure associated to $m^\beta_{\operatorname{inf}}$ equals the path $x_0=f_0f_1\dots\in E^\infty$, which clearly is an element of $\Ewan$.

We claim that $\Cinfb=\{m^\beta_{\operatorname{inf}}\}$. This follows from the fact that any $m\in \Cinf$ will satisfy
 $$
 m(v_n)/m(v_{n+1})=m^\beta_{\operatorname{inf}}(v_n)/m^\beta_{\operatorname{inf}}(v_{n+1})
 $$
 for all $n\geq 0$, which together with the conditions
 $\sum_{v\in E^0}m(v)=\sum_{v\in E^0} m^\beta_{\operatorname{inf}}(v)=1$ implies $m=m^\beta_{\operatorname{inf}}$.

It follows from Theorem \ref{thm:infty-states} that every ground state is a KMS$_\infty$ state.

We can summarize the preceding analysis in the following result.

\begin{theorem}\label{thm:exmotivating} Let $E$ be the graph described in \eqref{graph-motivating}. Consider $\mathcal{T}C^*(E)$ endowed with its gauge action. Then the KMS$_\beta$ states for $\beta\in [0,\infty)$  and the ground states of $\mathcal{T}C^*(E)$   are given as follows:
 \begin{enumerate}
\item if $\beta=0$, then $D^\beta$ consists of the single conservative function $m_{v_0}^0$;
 \item if $\beta\in (0,\ln 2)$, then $D^\beta=\Cf \sqcup \Cinfb$, where $\{m_{v_n}^\beta\}_{n\in \No}$ are all the extreme  points of $\Cf$ and $\Cinfb$ consists of the single dissipative function $m^\beta_{\operatorname{inf}}$;
 \item if $\beta\ge\ln 2$ we have $D^\beta=\Cf$, and the extremal KMS$_\beta$ states are $\{m_{v_n}^\beta\}_{n\in \No}$;
 \item the extreme points of the set $D^{\operatorname{gr}}$ of ground states are $\{m_{v_n}^{\operatorname{gr}}\}_{n\in \No}$, with $m_{v_n}^{\operatorname{gr}}$ as given in \eqref{eq:ground-pointmass}, and every ground state is a KMS$_\infty$ state.
 \end{enumerate}
\end{theorem}

Since $E^0_\reg=E^0$, the only KMS$_\beta$ and ground states that descend to $C^*(E)$ are the ones corresponding to $m_{v_0}^0$ and $m^\beta_{\operatorname{inf}}$, $\beta\in (0,\ln 2)$.

Note that we may describe the support of the measure associated to $m^\beta_{\operatorname{inf}}$  as an equivalence class for infinite paths, as follows. Given $x_1,x_2\in E^\infty$ we say that $x_1$ and $x_2$ are \emph{tail-equivalent} and write $x_1\sim_{\operatorname{tail}} x_2$ if there exist $x\in E^\infty$ and $u_1,u_2\in E^*$ such that $x_1=u_1x$ and $x_2=u_2x$. Thus the measure associated to $m^\beta_{\operatorname{inf}}$ has support $\{x\in E^\infty\mid x\sim_{\operatorname{tail}}x_0\}$.

Notice that similar to what we saw in Example \ref{ex:dis}, the sequence $\{m^\beta_{v_n}(v_k)\}_{n}$ converges to $m^\beta_{\operatorname{inf}}(v_k)$ for each $k\in \No$. This suggests that it may be possible in general to describe elements in $\Cinfb$ as pointwise limits, in appropriate sense, of elements in $\Cf$. Such a description for arbitrary graphs would be very interesting.
\end{example}

\begin{remark}
	Notice that the graphs considered in Example \ref{ex:dis} and Example~\ref{ex:motivating} are not strongly connected. Klaus Thomsen has shown us an example of a strongly connected graph for which $\Cinfb$ is non-empty.
	
	In the next example we will see that a small change in the graph of Example~\ref{ex:motivating} produces a (still not connected) graph where $\Ewan \neq\emptyset$ and yet $\Cinfb=\emptyset$. At the current stage we do not know what sort of additional information is needed in order to ensure that $\Cinfb$ is non-empty.
\end{remark}

\begin{example} \label{ex:b}
This example is a variation of Example~\ref{ex:motivating} where we add one more loop to $v_0$.  The graph is given by:
\begin{equation}
  \xymatrix{
  {v_0} \ar@(ul,ur)[]|{e_0} \ar@(dr,dl)[]|{d_0} \ar[r]^{f_0} &{v_1}\ar@(ul,ur)[]|{e_1} \ar[r]^{f_1} &{v_2}\ar@(ul,ur)[]|{e_2} \ar[r]^{f_2}&{\dots}
  }
\end{equation}
Let $R=\emptyset$ and $N:E^1\to (1,\infty)$ be $N(e_n)=N(f_n)=N(d_0)=\exp(1)$  for all $n\in \No$.

Proposition~\ref{prop:nice-graphs} implies that $D^\beta=\emptyset$ for $\beta<\ln 2$, $D^\beta=D^\beta_{\operatorname{inf}}$ when $\beta=\ln 2$, and that $D^\beta=\Cf=\conv\{m^\beta_v\mid v\in E^0\}$ for $\beta>\ln 2$.

Let $m\in D^{\ln 2}$. Repeated applications of Lemma~\ref{lem:mineq} show that in case $m(v_0)=0$, then $m(v_n)=0$ for all $n\geq 0$, a fact that would contradict \eqref{item:m1}. Thus $m(v_0)\neq 0$. Since $m(v_0)\ge m(v_0)+\frac{1}{2}m(v_1)$ by \eqref{item:m3}, we must have that $m(v_1)=0$. It thus follows from Lemma~\ref{lem:mineq} that $m(v_n)=0$ for all $n\geq 1$. Hence $D^{\ln 2}=\{m_{v_0}^{\ln 2}\}$.

Thus $\Cinfb=\emptyset$ for all $\beta$ although for example $x=f_0f_1f_2\dots $ is a wandering path.
\end{example}

\begin{example} \label{ex:c}
	We briefly show how our analysis recovers the known results valid for the Cuntz algebra $\mathcal{O}_n$ and the Toeplitz-Cuntz algebra $\mathcal{TO}_n$ for $n\geq 2$. The graph in question has $E^0=\{v\}$ and $E^1=\{e_1, e_2,\dots, e_n\}$, where $s(e_i)=r(e_i)=v$ for all $i=1, \dots, n$. Thus we are dealing with a single vertex and $n$ loops of length one based at $v$, and in particular $E^0_{\operatorname{reg}}=\{v\}$. We let $N$ be the gauge action, thus $N(e_i)=\exp(1)$ for all $i$ and $N(v)=1$.
	
Let $R=\emptyset$. Proposition~\ref{prop:nice-graphs} implies that $D^\beta=\emptyset$ for $\beta<\ln n$, $D^\beta=D^\beta_{\operatorname{inf}}$ when $\beta=\ln n$, and  $D^\beta=\Cf=\{m^\beta_v\}$ for $\beta>\ln n$. Since $E^0_{\ln n\text{-crit}}=\{v\}$, it follows that $D^{\ln n}=\{m^{\ln n}_v\}$.
Finally we see that $D^{\operatorname{gr}}=\{m^{\operatorname{gr}}_{v}\}$. It follows from Theorem \ref{thm:infty-states} that $m^{\operatorname{gr}}_{v}$ is a KMS$_\infty$ state.
	
If $R=E^0=\{v\}$, then $D^{\operatorname{gr}}=D^\beta=\emptyset$ unless $\beta=\ln n$, in which case $D^\beta=\{m^{\ln n}_v\}$.
\end{example}

\begin{example}\label{eq:stable-TOn}
Consider the graph with one countable ``straight line'' ending in a vertex $v_1$ that emits $n$ distinct loops:
\begin{equation}\label{graph-motivating2}
  \xymatrix{
 &{\dots}\ar[r]^{f_3} &{v_3} \ar[r]^{f_2} &{v_2} \ar[r]^{f_1} &{v_1}\ar@(ul,ur)[]|{e_1} \ar@(ur,dr)[]|{e_j} \ar@(dl,dr)[]|{e_n}
  }
\end{equation}
We let $N:E^1\to (1,\infty)$ be given by $N(e)=\exp(1)$ for $e\in E^1$. Then
$$
v_kE_s^*v_k=\begin{cases}\{e_1, e_2,\dots, e_n\}&\text{ if }k=1\\
\emptyset &\text{ if }k\geq 2,
\end{cases}
$$
$E_a^*v_k=\{f_n\dots f_k\mid n\geq k\}\cup \{v_k\}$ for $k\geq 2$, and
$E_a^*v_1=\{f_n\dots f_1\mid n\geq 1\}\cup \{v_1\}\cup v_1E_s^*v_1$.

We next determine the partition functions associated to this graph. Given $\beta\in [0,\infty)$,
$$
Z^s_{v_1}(\beta)=\sum_{j=1}^n N(e_j)^{-\beta}= n\exp(-\beta)
$$
and
\begin{align*}
Z^a_{v_1}(\beta)
&=Z^s_{v_1}(\beta)+1+\sum_{u=f_n\dots f_1, n\geq 1}N(u)^{-\beta}\\
&= n\exp(-\beta) + 1 + \sum_{m\geq 0}(\exp(-\beta))^{m+1}.
\end{align*}
Hence $Z^a_{v_1}(\beta)<\infty$ if and only if $\beta>0$ and $Z^s_{v_1}(\beta)=1$ precisely when $\beta=\ln n$.

For $k\geq 2$, $Z^s_{v_k}(\beta)=0$ and $Z^a_{v_k}(\beta)=\sum_{m\geq 0}(\exp(-\beta))^{m+1}+1$. The sets of regular and critical points are listed in the following table:

\begin{table}[h]
\begin{center}
\begin{tabular}{c||c|c|c|c}
$\beta$ &0& $(0,\ln n)$ &$\ln n$ & $(\ln n, \infty)$\\
\hline \hline
$\Ereg$ & $\emptyset$&$E^0\setminus\{v_1\}$ & $E^0\setminus\{v_1\}$& $E^0$\\
\hline
$\Ecrit$& $\emptyset$&$\emptyset$& $\{v_1\}$& $\emptyset$
\end{tabular}
\end{center}
\end{table}
%\vskip 0.5cm
%For comparison, the similar table for the graph giving $\mathcal{TO}_n$, i.e. $E$ with the countable "straight-line" removed, the %corresponding table is:

%\vskip 0.2cm
%\begin{tabular}{c||c|c|c}
%$\beta$ & $(0,\log(n))$ &$\log(n)$ & $(\log(n), \infty)$\\
%\hline \hline
%$\Ereg$ & $\emptyset$ & $\emptyset$& $E^0=\{v_1\}$\\
%\hline
%$\Ecrit$& $\emptyset$& $\{v_1\}$& $\emptyset$
%\end{tabular}
%\vskip 0.2cm
Every infinite path passes through $v_1$ infinitely many times so $\Ewan=\emptyset$ and $\Cinfb=\emptyset$ for all $\beta$. Hence we can characterize all the extremal KMS$_\beta$ states for  $\mathcal{T}C^*(E)$: $D^0=\emptyset$, $D^\beta=D^\beta_{\operatorname{fin}}=\conv\{m_{v_k}\mid k\geq 2\}$ when $\beta\in (0,\ln n)$, $D^{\ln n}_{\operatorname{fin}}=\conv\{m_{v_k}\mid k\geq 2\}$, $D^{\ln n}_{\operatorname{inf}}=D^{\ln n}_{\operatorname{con}}=\{m^\beta_{v_1}\}$, and $D^\beta=\Cf=\{m_{v_k}\mid k\geq 1\}$ when $\beta\in (\ln n,\infty)$. Furthermore, $D^{\operatorname{gr}}=\conv\{m^{\operatorname{gr}}_{v_k}\mid k\le 1\}$, and every ground state is a KMS$_\infty$ state by Theorem \ref{thm:infty-states}.

Since $E^0_\reg=E^0$, the only KMS and ground states that descend to $C^*(E)$ are $m^{\ln n}_{v_1}$ and  $m^{\operatorname{gr}}_{v_1}$.

We present the extremal KMS states by comparison with $\mathcal{O}_n$ and $\mathcal{TO}_n$:

\begin{table}[h]
	\begin{center}
		\begin{tabular}{c||c|c|c|c|c}
		$\beta$ & $0$ &$(0,\ln n)$ & $\ln n$& $(\ln n, \infty)$&gr\\
		\hline \hline
		$\mathcal{T}C^*(E)$& $\emptyset$ & $\{m^\beta_{v_k}\mid k\geq 2\}$& $\{m^\beta_{v_k}\mid k\geq 1\}$& $\{m^\beta_{v_k}\mid k\geq 1\}$&$\{m^{\operatorname{gr}}_{v_k}\mid k\ge 1\}$\\
		\hline
		$C^*(E)$& $\emptyset$ & $\emptyset $& $\{m^\beta_{v_1}\}$& $\emptyset$ & $\{m^{\operatorname{gr}}_{v_1}\}$\\
		\hline
		$\mathcal{TO}_n$ & $\emptyset$ & $\emptyset$ & $\{m^\beta_v\}$ & $\{m^\beta_v\}$&$\{m^{\operatorname{gr}}_v\}$\\
		\hline
		$\mathcal{O}_n$ & $\emptyset$ & $\emptyset$ & $\{m^\beta_v\}$ & $\emptyset$&$\emptyset$
		\end{tabular}
	\end{center}
\end{table}

Note that the graph consisting of a countable straight line underlies the algebra $\mathcal{K}$ of compact operators on a separable
Hilbert space. In particular, $\mathcal{T}C^*(E)\cong \mathcal{TO}_n \otimes \mathcal{K}$.
\end{example}

\begin{example} \label{ex:d}
	Finally we present an example of a ground state which is not a KMS$_\infty$ state.
	
	Let $E^0=\{v\}$ and $E^1=\{e_1,e_2,\dots\}$. Then $E^0_\reg=\emptyset$ and $\mathcal{T}C^*(E)=C^*(E)=\mathcal{O}_\infty$. Let $N:E^1\to (1,\infty)$ be given by $N(e)=\exp(1)$ for $e\in E^1$. Clearly $\Eequ=\emptyset$ for all $\beta$ and $D^{\operatorname{gr}}=\{m_v^{\operatorname{gr}}\}$. It follows that there are no KMS states and that the ground state corresponding to $m_v^{\operatorname{gr}}$ is not a KMS$_\infty$ state.
\end{example}

\end{document}